\documentclass[11pt]{article}
\usepackage{latexsym,amsfonts,amsthm,amsmath,amscd,amssymb,color,url}
\usepackage[dvips]{graphicx}

\newtheorem{lemma}{Lemma}[section]

\newtheorem{rem}[lemma]{Remark}
\newtheorem{prop}[lemma]{Proposition}

\newtheorem{defn}[lemma]{Definition}

\newcommand\matR{{\mathbb{R}}}

\newcommand\Dhat{{\widehat{D}}}

\newcommand\Omegatil{{\widetilde\Omega}}
\newcommand\alphatil{{\widetilde\alpha}}
\newcommand\betaztil{{\widetilde\beta}_0}

\newcommand\Ctil{{\widetilde C}}

\renewcommand{\hbar}{{\overline{h}}}

\newfont{\Got}{eufm10 scaled 1200}

\newcommand{\mycap} [1] {\caption{\footnotesize{#1}}}

\newcommand\calP{{\mathcal P}}

\begin{document}

\title{Algorithmic simplification of knot diagrams:\\
new moves and experiments}

\author{Carlo~\textsc{Petronio\thanks{Partially supported by the Italian FIRB project ``Geometry and topology of low-dimensional manifolds'' RBFR10GHHH.}}\and Adolfo~\textsc{Zanellati}}

\maketitle

\begin{abstract}
\noindent
This note has an experimental nature and contains no new theorems.

We introduce certain moves for classical knot diagrams
that for all the very many examples we have tested them on give
a \emph{monotonic complete simplification}.
A \emph{complete simplification} of a knot diagram $D$ is a sequence of moves that
transform $D$ into a diagram $D'$ with the minimal possible number of crossings for the isotopy class
of the knot represented by $D$.
The simplification is \emph{monotonic} if the number of crossings never increases along the sequence.
Our moves are certain $Z_{1,2,3}$
generalizing the classical Reidemeister moves $R_{1,2,3}$, and
another one $C$ (together with a variant $\Ctil$) aimed at detecting whether
a knot diagram can be viewed as a connected sum of two easier ones.

We present an accurate description of the moves and several results of
our implementation of the simplification procedure based on them, publicly available on the web.

\smallskip


\smallskip

\end{abstract}

\noindent
This paper describes constructions and experimental results originally due to the second named author, that
were formalized, put into context, and generalized in collaboration with the first named author.
No new theorem is proved.

We introduce new combinatorial moves $Z_{1,2,3}$ on knot diagrams,
together with another move $C$ and a variant $\Ctil$ of $C$,
that perform very well in the task of simplifying diagrams.
Namely, as we have checked by implementing the moves in~\cite{zanweb} and applying them
to a wealth of examples, the moves are apparently very efficient in carrying out
a \emph{monotonic complete simplification}.
Here by \emph{complete simplification} of a knot diagram $D$ we mean a sequence of moves that
transform $D$ into a diagram $D'$ with the minimal possible number of crossings for the isotopy class
of the knot represented by $D$.
The simplification is \emph{monotonic} if the number of crossings never increases along the sequence.

The moves $Z_{1,2,3}$ extend the classical
Reidemeister moves $R_{1,2,3}$, while
$C$ and $\Ctil$ are aimed at detecting whether a diagram can be viewed as the connected
sum of two easier ones. In practice, the simplification procedure based on the
moves $Z_{1,2,3}$ only is already quite powerful, since it allows for instance
to monotonically untangle most of the known hard diagrams of the unknot.
However the moves $Z_{1,2,3}$ only are not sufficient in general, as an example
provided to us by the referee shows, but this example is easily dealt with using
the move $C$. We do not know whether the use of all the moves $Z_{1,2,3}$, $C$ and $\Ctil$
allows a monotonic complete simplification of any knot diagram, but if this were
the case one would have an algorihtm to compute the crossing number of a knot,
and in particular to detect knottedness.

\medskip

In the history of knot theory quite some energy has been devoted to
the problem of effectively detecting whether or not a knot diagram represents the unknot,
and more generally of computing the crossing number of a link starting from an arbitrary
diagram representing it.
The point here is that, using only Reidemeister moves,
one may have to temporarily increase the number
of crossings before reducing the link to a minimal crossing diagram.
This phenomenon occurs both for the unknot and for more general knots and links,
as explained below. Many solutions of the
unknotting problem, with various degrees of suitability for actual implementation,
have been obtained over time, see Birman and Hirsch~\cite{bihi},
Galatolo~\cite{galatolo}, Hass-Lagarias~\cite{HaLa},
Dynnikov~\cite{dyn} (building on Birman and Menasco~\cite{bime1, bime2} and
Cromwell~\cite{crom}), and Lackenby~\cite{lack3}.
Dynnikov's technique also leads to the solution of other
computational problems in link theory (see also Kazantsev~\cite{kaz} for a further development), but apparently not
a method for calculating the
crossing number. The result of Lackenby provides an upper bound on the number
of Reidemeister moves needed to untangle a diagram of the unknot, so it is
a fundamental one from a theoretical viewpoint, but it is of impractical use.
See also \cite{ChuMi, Fol, EnKa} for some contributions more
focused on algorithmic efficiency.

Concerning the crossing number,
a vast and deep literature exists on it,
see for instance \cite{Ohy, king, Gru, ChuLi, Hoste, HoSha, Sto} and the recent
fundamental results of Lackenby~\cite{lack1, lack2},
but no effective algorithm to compute it is available to our knowledge.
We note that the achievement of Coward and Lackenby~\cite{CowLac} yields a theoretical algorithmic procedure to decide
whether two diagrams represent the same link, so in principle it allows one to compute
the crossing number in an indirect way, namely by comparing a diagram
to \emph{all} those having fewer crossings. The moves introduced in this paper are not
proved to actually yield an algorithm to compute the crossing number in general,
but in practice they seem to work very effectively.
We mention here that the simplification algorithm based
on our moves $Z_{1,2,3}$, $C$ and $\Ctil$ applies not
only to knots but also to multi-component links, but
our implementation~\cite{zanweb} is currently restricted to knots, so we have
no evidence of efficiency for links so far.

\bigskip

The structure of the paper is as follows. In Section~\ref{Z123:section} we introduce the three moves
$Z_{1,2,3}$ that extend the classical Reidemeister moves (see, \emph{e.g.},~\cite{rolfsen}),
and we describe an algorithm based on these moves aimed at simplifying link diagrams.
Then in Section~\ref{experimental:section} we describe many experimental applications of our
algorithm to famously hard knot diagrams, displaying its remarkable practical performance.
Finally, in Section~\ref{Cmove:section} we introduce the further moves $C$ and $\Ctil$,
showing that they sometimes make the simplification task much faster,
and, more importantly, verifying that in some cases they allow simplifications that
the moves $Z_{1,2,3}$ only are not capable of realizing.

\bigskip\noindent\textsc{Acknowledgements}\quad
We thank Malik Obeidin for indicating to us the example shown in Fig.~\ref{malik:fig} and for
confirming the minimality of the 28-crossings diagram referred to at the end of Section~\ref{experimental:section}.
We also warmly thank the referee for having suggested to test our algorithm on the diagrams
of Figg.~\ref{fulltwist:fig},~\ref{138:fig} and~\ref{120:fig} (see also
Remark~\ref{old:version:rem}).

\section{Generalized Reidemeister moves and\\
the reduced simplification algorithm}\label{Z123:section}
In this section we introduce the moves $Z_{1,2,3}$
and we describe the diagram simplification algorithm based on them.

\paragraph{Notation for links and the crossing number}
We will define our moves in a formal abstract way, and we will also illustrate them pictorially.
To do this, we define a \emph{link} $L$ to be a tame embedding $\ell:\mathop{\sqcup}\limits_{i=1}^p S_i^1\to S^3$, where
each $S^1_i$ is a circle, $S^3=\matR^3\cup\{\infty\}$ and $\ell$ avoids $\infty$, with $\ell$ viewed up to isotopy in $S^3$.
We then define a \emph{link diagram} $D$ to be an immersion $p:\mathop{\sqcup}\limits_{i=1}^p S_i^1\to S^2$
with normal double points (crossings) as only singularities, where
$S^2=\matR^2\cup\{\infty\}$ and $p$ avoids $\infty$, together with the usual under-over
indication at each crossing (no specific notation is employed for this indication).
Here $D$ is viewed up to isotopy on $S^2$ and it determines a link $[D]$.

We denote by $c(D)$ the number of crossings of a diagram $D$, and by $c(L)$ the minimum
of $c(D)$ over the diagrams $D$ with $[D]=L$. This minimum is called the
\emph{crossing number} of $L$. If $[D]=L$ and $c(D)=c(L)$ we say that $D$ is \emph{minimal}.
The general aim of this paper is to describe combinatorial moves capable of completely
simplifying any given diagram $D$ in a monotonic fashion, namely transforming $D$ into
a minimal diagram $D'$ with $[D']=[D]$ without ever increasing $c$.

\paragraph{Graphic conventions}
A link diagram $D$ is drawn as usual, and in our figures of a move $\mu$ we always adopt the following
conventions:
\begin{itemize}
\item A thin solid line is part of $D$;
\item A slightly thicker line is also part of $D$, highlighted for the role it
plays in the definition of $\mu$;
\item A very thick line represents a portion of the plane where several strands of $D$ can
appear and cross each other; often these thick lines are drawn so as to merely suggest the possible
behaviors of $D$, the exact conditions to be met being described in the formal definition;
\item A dashed line is always transverse to $D$ and not part of $D$, with its ends (if any)
on $D$ but not at crossings;
\item A gray shading is used to highlight a planar region
playing a role in the definition of $\mu$.
\end{itemize}

If $\alpha\subset S^1_i$ is a closed arc we call \emph{extension} of $p(\alpha)$ a curve $p(\alphatil)$
where $\alphatil$ is an arc that contains $\alpha$ in its interior. A \emph{germ of extension} is one
of the two components of $p(\alphatil\setminus\alpha)$, viewed up to inclusion.

\paragraph{The move $Z_1$}
Let $\alpha\subset S_i^1$ be a closed arc such that $p|_\alpha$ is a simple closed
curve, and let $\Omega$ be the component of $S^2\setminus p(\alpha)$
that does not contain the germs of extensions of $p(\alpha)$.
Let $\beta_1,\ldots,\beta_N$ be the components (each an arc or a circle) of $p^{-1}(\Omega)$,
and suppose we can assign them labels $\lambda_1,\ldots,\lambda_N$ in $\{U,O\}$ so that:
\begin{itemize}
\item If $\beta_j$ is an arc then $p(\beta_j)$ is over $p(\alpha)$ at both ends if $\lambda_j=O$,
and under $p(\alpha)$ if $\lambda_j=U$;
\item If $\lambda_j=U$ and $\lambda_k=O$ then $p(\beta_j)$ is under $p(\beta_k)$ where they cross.
\end{itemize}
We then call $Z_1$ the move that consists of collapsing $\alpha$ to a point (see Fig.~\ref{Z1:fig}).
\begin{figure}
    \begin{center}
    \includegraphics[scale=.6]{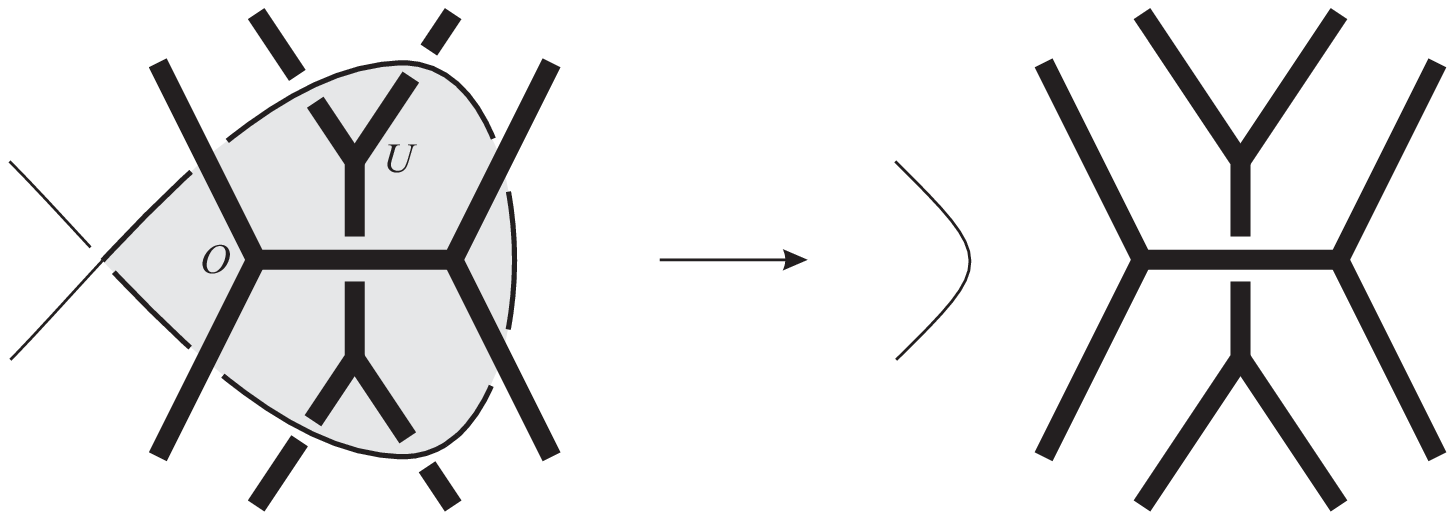}
    \end{center}
\mycap{The move $Z_1$. The thick arc on $D$ is $p(\alpha)$, and the gray region is $\Omega$. \label{Z1:fig}}
\end{figure}
Of course $Z_1$ preserves the link type, and the Reidemeister move $R_1$ is a special case of $Z_1$.

Figure~\ref{Z1example:fig}-left illustrates the \emph{wrong} choice of $\Omega$ as a component of $S^2\setminus p(\alpha)$.
An application of $Z_1$ is given in Fig.~\ref{Z1example:fig}-right.
\begin{figure}
    \begin{center}
    \includegraphics[scale=.6]{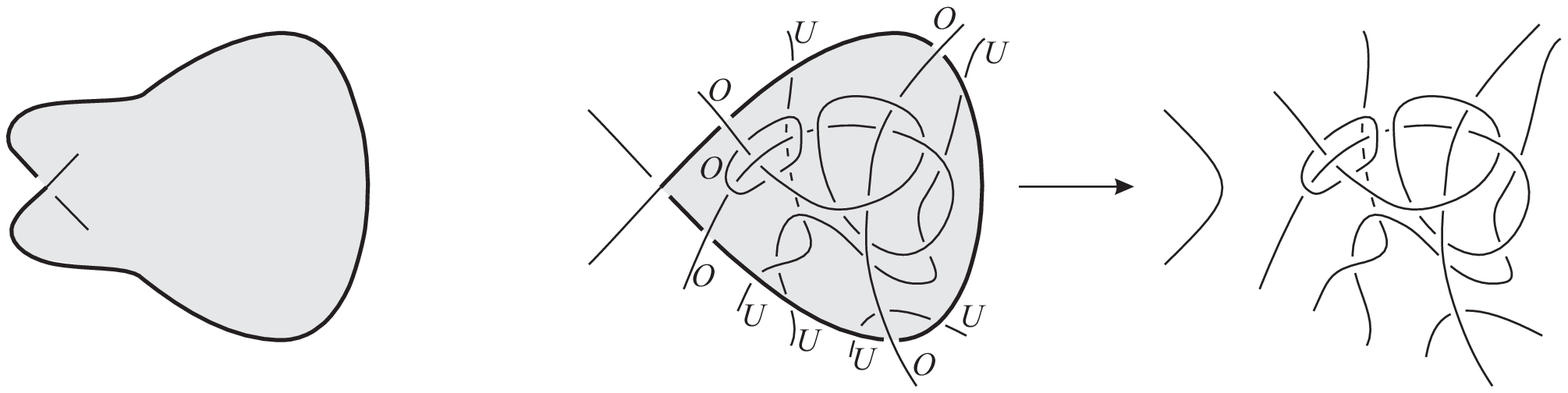}
    \end{center}
\mycap{Left: the wrong choice of $\Omega$ for $Z_1$. Right: an instance of $Z_1$.\label{Z1example:fig}}
\end{figure}

\paragraph{The move $Z_2$}
For $\lambda=U,O$ let $\alpha_\lambda\subset S^1_{i_\lambda}$ be a closed arc such that
$p|_{\alpha_\lambda}$ is a simple curve. Suppose that $p(\alpha_U)$ and $p(\alpha_O)$ have the
same ends at two crossings of $D$, and that $p(\alpha_U)$ is under $p(\alpha_O)$ at both.  Set $\alpha=\alpha_U\cup\alpha_O$,
let $\Omega$ be one of the components of $S^2\setminus p(\alpha)$, and assume that $\Omega$
does not contain any of the four germs of extensions of $p(\alpha)$.
Let $\beta_1,\ldots,\beta_N$ be the components (each a arc or a circle) of $p^{-1}(\Omega)$,
and suppose it is possible to assign
them labels $\lambda_1,\ldots,\lambda_N$ in $\{U,O\}$ so that
precisely the same conditions as in the definition of the move $Z_1$ are met.
Under these assumptions we call $Z_2$ the move that consists of interchanging $p|_{\alpha_U}$ and $p|_{\alpha_O}$
and pulling apart their ends (see Fig.~\ref{Z2:fig}).
\begin{figure}
    \begin{center}
    \includegraphics[scale=.6]{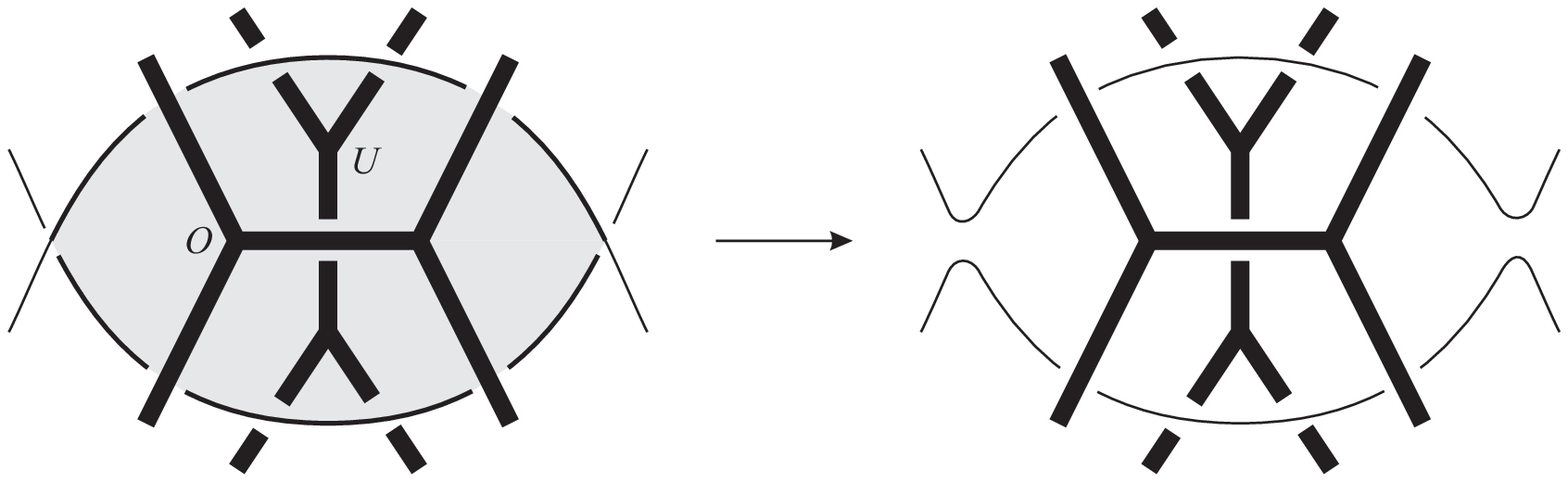}
    \end{center}
\mycap{The move $Z_2$. The thick arcs on $D$ are $p(\alpha_{U,O})$, and the gray region is $\Omega$. \label{Z2:fig}}
\end{figure}
Of course $Z_2$ preserves the link type, and the Reidemeister move $R_2$ is a special case of $Z_2$.

\begin{rem}\label{label:choices:rem1}
\emph{In the description of both the moves $Z_1$ and $Z_2$,
the labels for the $\beta_j$'s that are arcs are fixed by the first
condition, while both labels $U$ and $O$ should be tested for each
$\beta_j$ being a circle, verifying whether some choice
allows the second condition to be met.}
\end{rem}

\paragraph{The move $Z_3$}
Let $\alpha\subset S^1_i$ be a closed arc such that $p|_\alpha$
is a simple curve with its ends not being crossings of $D$.
Suppose that $p(\alpha)$ is over (respectively, under) at all the crossings it contains.
Let $\gamma$ be a simple arc in $S^2$ with the same ends as $p(\alpha)$
but otherwise disjoint from it, and transverse to $D$ (including at its ends).
Then the move $Z_3$ consists of replacing $p|_\alpha$ by $\gamma$ and stipulating that
$\gamma$ is over (respectively, under) at all the crossings it contains.
Figure~\ref{Z3:fig} shows this move for the case of an overarc and
\begin{figure}
    \begin{center}
    \includegraphics[scale=.7]{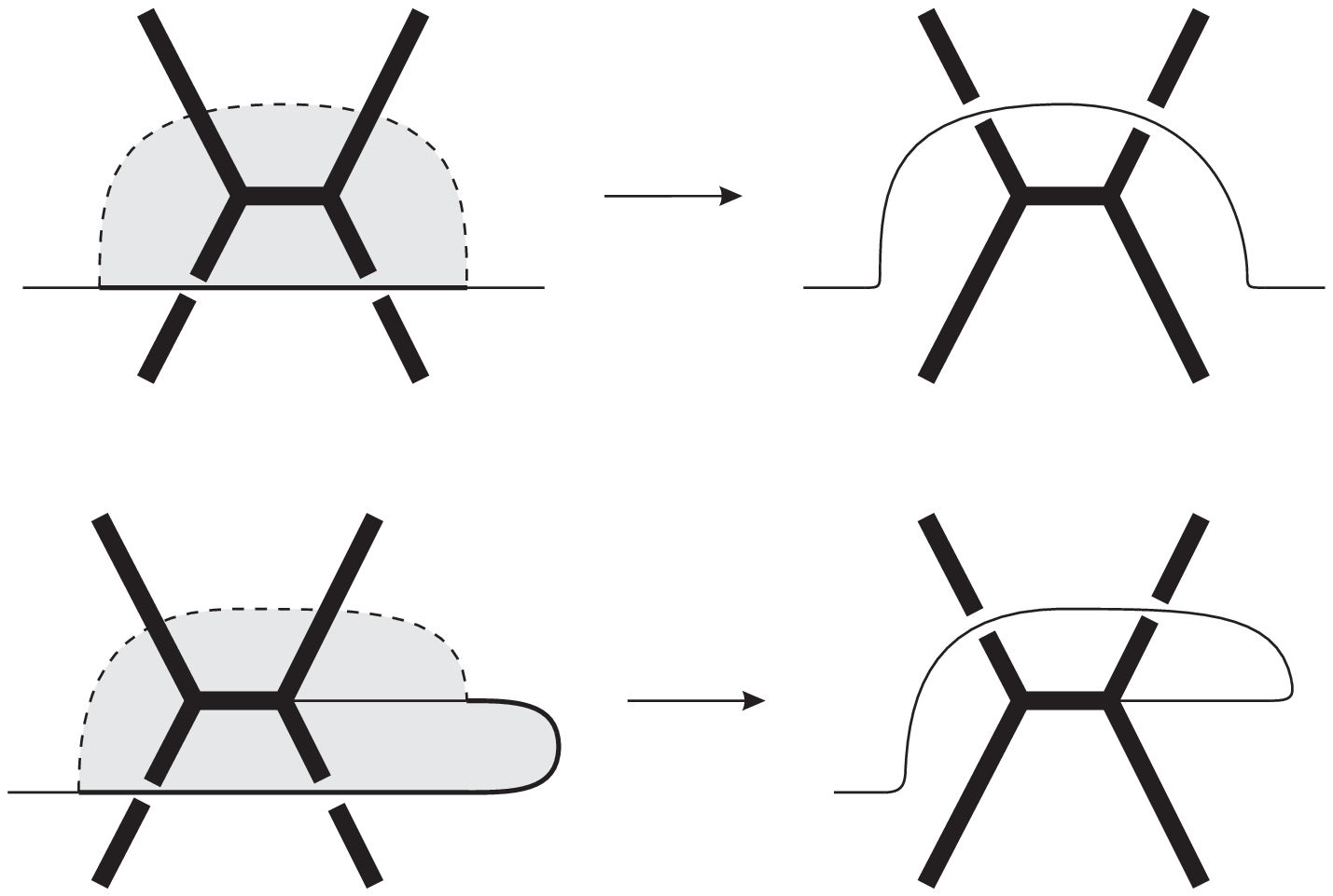}
    \end{center}
\mycap{Move $Z_3$ for an overarc, in its two possible versions.
The thick arc is $p(\alpha)$, and the gray region $\Omega$ is actually not
necessary to define the move (see the text).\label{Z3:fig}}
\end{figure}
displays a small subtlety: two slightly
different situations can arise, namely the two ends of $\gamma$ can lie on the
same side of $p(\alpha)$, as in Fig.~\ref{Z3:fig}-top, or on opposite sides,
as in Fig.~\ref{Z3:fig}-bottom. In other words, if $\Omega$ is
one of the two components of $S^2\setminus (p(\alpha)\cup\gamma)$, we see
that $\Omega$ can contain an even (0 or 2) or odd (1) number of germs
of extensions of $p(\alpha)$ ---the two even cases are the same up to the choice of $\Omega$.

We will now introduce some restrictions on $Z_3$. To this we note that for a split
link $L_1\sqcup L_2$ we have $c(L_1\sqcup L_2)=c(L_1)+c(L_2)$, so
we can rule out this case and always assume that the regions of $S^2\setminus D$ are
topological discs. We then denote by $\Dhat$ the dual planar graph, and
we stipulate that the move $Z_3$ can be applied only
if both the following hold:
\begin{itemize}
\item (\emph{maximality}) $\alpha$ is a maximal overarc or underarc of $D$;
\item (\emph{minimality}) $\gamma$ defines a minimal path in
the graph $\Dhat$ dual to $D$.
\end{itemize}
Note that the minimality condition is met even if $\gamma$ does cross $D$ at all; in this case,
the induced path in $\Dhat$ is a point.

\begin{rem}
\emph{As shown in Fig.~\ref{R3isZ3:fig},
\begin{figure}
    \begin{center}
    \includegraphics[scale=.6]{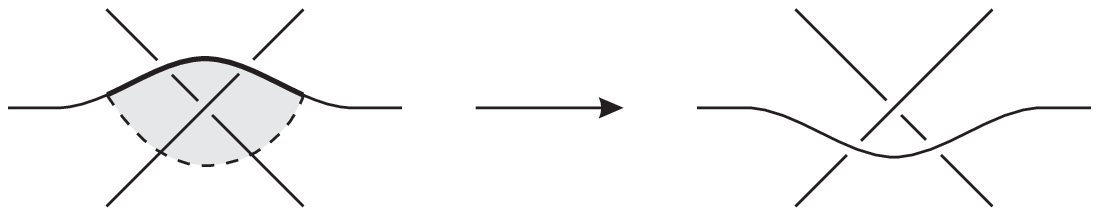}
    \end{center}
\mycap{An $R_3$ move is a $Z_3$ move.\label{R3isZ3:fig}}
\end{figure}
a Reidemeister move  $R_3$ can always be realized as a $Z_3$ (but perhaps not satisfying
maximality and minimality).}
\end{rem}

\paragraph{The reduced simplification algorithm}
A move $D\to D'$ on link diagrams is called \emph{decreasing} if $c(D')<c(D)$,
\emph{horizontal} if $c(D')=c(D)$, and \emph{non-increasing} if $c(D')\leqslant c(D)$.
A sequence $D=D_0\to D_1\to \ldots \to D_k\to D'$ of moves is called \emph{monotonic} if
each move $D_{i-1}\to D_i$ is non-increasing, and \emph{strictly monotonic} if each
$D_{i-1}\to D_i$ is decreasing.

We now describe the reduced version of the algorithm implemented in~\cite{zanweb}
(see Section~\ref{Cmove:section} for the complete version):

\bigskip

\noindent \textbf{Procedure} $\calP$\qquad If there is
a sequence $D=D_0\to D_1\to \ldots \to D_k\to D'$
with each $D_{i-1}\to D_i$ a horizontal $Z_3$ move and
$D_k\to D'$ a decreasing $Z_{1,2,3}$ move, then apply the procedure to $D'$. If not, output $D$.

\bigskip

Two remarks about this procedure are in order:
\begin{itemize}
\item While applying procedure $\calP$ to any given $D$ one can restrict to
sequences of horizontal $Z_3$ moves $D=D_0\to D_1\to \ldots \to D_k$
such that the $D_i$'s are pairwise distinct, so a finite number of such sequences
must be inspected;
\item After one sequence $D=D_0\to D_1\to \ldots \to D_k\to D'$ of moves
as in the definition of $\calP$ is found, the moves are immediately
performed and $\calP$ proceeds with $D'$, so $D$ is never considered again.
\end{itemize}

We say that $\calP$ succeeds on a link diagram $D$ if it allows a complete simplification
of $D$, namely, if it turns $D$ into a diagram realizing the crossing number of $[D]$.
Note that by definition this simplification is always monotonic, and it is strictly monotonic
if no horizontal $Z_3$ move is employed.

\begin{rem}\label{Z3:restrictions:rem}
\emph{By dropping the minimality and maximality restrictions on $Z_3$,
one could have a potentially more general simplification procedure, but
in practice this does not lead to any advantage. Similarly, one could inspect
all the possible initial sequences $D=D_0\to D_1\to \ldots \to D_k\to D'$,
and recursively proceed for all the diagrams $D'$ thus obtained, but again this
has no practical effect.}
\end{rem}

\begin{rem}\label{old:version:rem}
\emph{In practice procedure $\calP$ has proved extremely efficient in bringing
to their minimal status many complicated knot diagrams, see Section~\ref{experimental:section}.
In a previous version of this paper we were putting forward the conjecture that $\calP$ would
successfully do this for \emph{every} diagram, but now \emph{we know this conjecture is false},
because the referee provided us with a 120-crossing diagram of the unknot that does not
untangle via non-increasing $Z_{1,2,3}$ moves, shown in Fig.~\ref{120:fig} below.
In fact, when we wrote the first version
of the paper, we already had in our algorithm a further move $C$, described in
Section~\ref{Cmove:section} below. We knew examples where $C$ allowed a quicker monotonic
simplification than the $Z_{1,2,3}$ only, but we decided not to include $C$ in the description
of the algorithm because we had no example in which $C$ was essential.
The 120-crossing diagram is precisely such an example, because after one move $C$
it untangles very quickly via $Z_{1,2,3}$. The enhanced procedure that uses also $C$ will be
described in Section~\ref{Cmove:section}.}
\end{rem}

\section{Experimental evidence}\label{experimental:section}
In this section we explain the practical performance of the reduced procedure
$\calP$ on many knot diagrams that are known to be hard, and particularly
on many diagrams of the unknot that do not monotonically untangle via
Reidemeister moves. We will provide full details for some examples,
while we will confine ourselves to the essential information for other cases,
referring the reader to~\cite{zanweb}. This website contains a publicly accessible
implementation of our algorithm, with documentation and instructions.
In particular, our software allows the user to input diagrams and simplify them.

We emphasize here that we know that $\calP$ does not always succeed (see Remark~\ref{old:version:rem}
above) and that~\cite{zanweb} already implements the complete procedure described below in
Section~\ref{Cmove:section}.

\paragraph{Technical notes}
We have implemented~\cite{zanweb} procedure $\calP$ using
Visual Basic 6, at 32 bits, under Windows XP on a 2010 laptop, and most of the time it takes is
actually absorbed by the graphic handling of the diagrams ---the ``quick simplification''
procedure typically runs much faster. For practical reasons
our implementation is currently
limited to the case of knots, but the extension to the case of
multi-component links, possibly even split ones, presents no theoretical
difficulties and will be carried out in the future.

\bigskip

We will begin by showing in full detail how the famous
Culprit, Goeritz and Thistlethwaite diagrams of the unknot
(see for instance~\cite{HaLa, HeKa, KaLa}) untangle in a strictly monotonic
and very quick way using our $Z_{1,2,3}$ moves.
Then we will treat the Hass-Nowik knots~\cite{HaNo}, showing that the moves
$Z_{1,2,3}$ untangle them in linear time (while the Reidemeister moves do so only quadratically).
Next, we will describe the performance of our algorithm on many other diagrams
of the unknot, mentioning in particular cases where a monotonic untangling via $Z_{1,2,3}$ moves
succeeds but a strictly monotonic untangling impossible, and even cases
where more than one consecutive horizontal $Z_3$ move is unavoidable.
Finally, we will turn to non-trivial (and even composite) knots, showing
that our reduced algorithm is also successful (and extremely efficient) in bringing their diagrams
to a minimal crossing number status. We single out here the Kazantsev knot
diagram~\cite{kaz}, with 23 crossings, which is intractable using only the Reidemeister moves,
while our procedure $\calP$ monotonically reduces it to its minimal
status with 17 crossings via $6$ moves of type $Z_3$ (one of which is horizontal).
To conclude we will provide two examples of diagrams of non-trivial knots
that require more than one consecutive horizontal $Z_3$ move to reach a minimal status.

\paragraph{The Culprit knot}
In Fig.~\ref{culprit:fig} we show how procedure $\calP$ applies to the Culprit knot.
\begin{figure}
    \begin{center}
    \includegraphics[scale=.6]{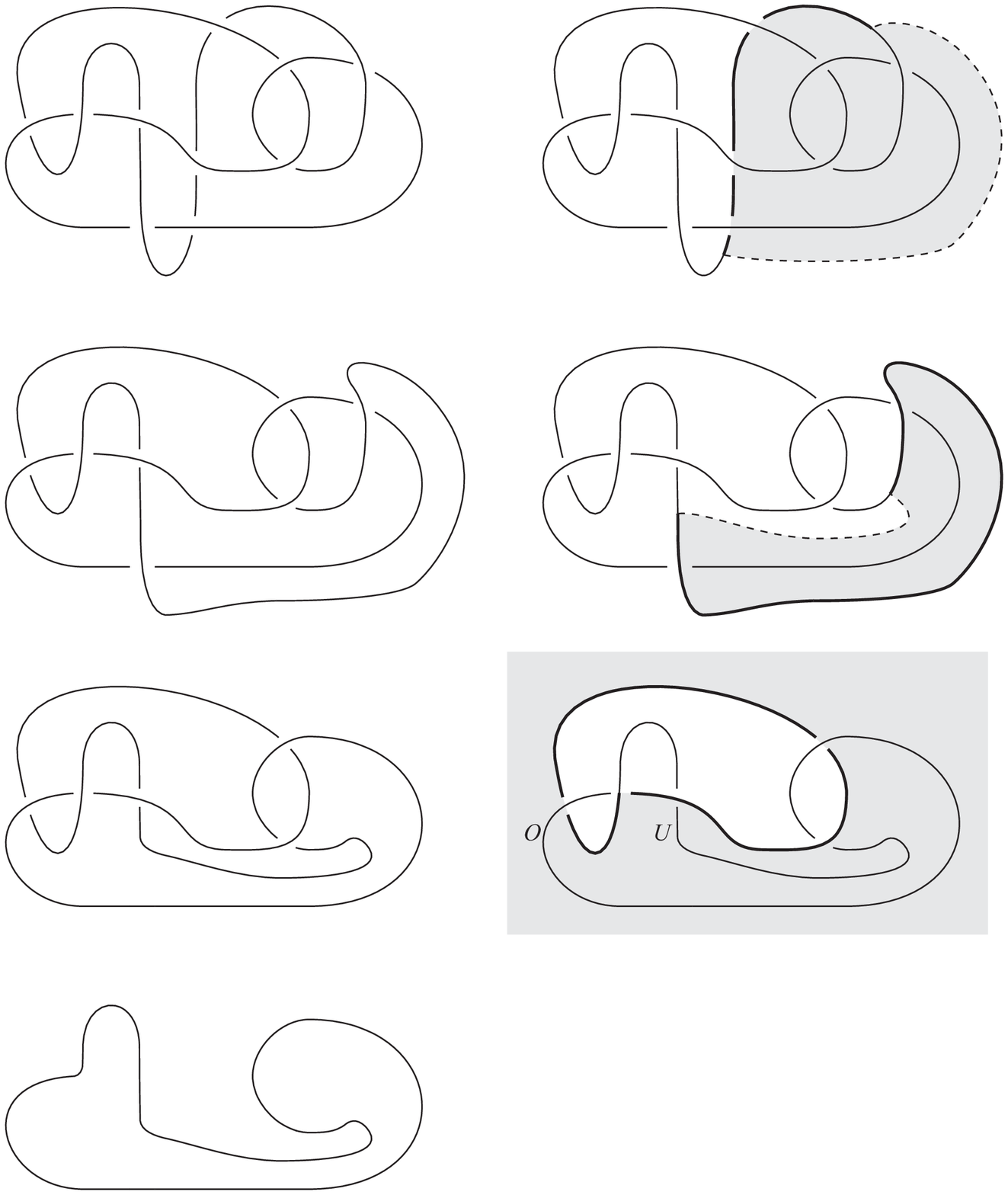}
    \end{center}
\mycap{Monotonic untangling of the Culprit knot via two $Z_3$ and one $Z_1$ moves.\label{culprit:fig}}
\end{figure}
In the top-left corner of the picture we show the original Culprit diagram, then on its right the
identification of a $Z_3$ move that applies to it, and below it the application of this move.
The rest of the picture is similarly organized, with one more $Z_3$ move and one $Z_1$
move that completely untangle the diagram. We make the following remarks:
\begin{itemize}
\item No horizontal $Z_3$ move is needed in this case;
\item The second $Z_3$ move applied could equivalently be described as a $Z_2$ (and, actually, an $R_2$) move;
moreover, two other $Z_2$ moves could be applied instead of it, see
Fig.~\ref{culpritbis:fig}, and both would also quickly lead to the untangling; in all the subsequent examples
we will refrain from showing alternative simplification moves, even when many exist;
\begin{figure}
    \begin{center}
    \includegraphics[scale=.6]{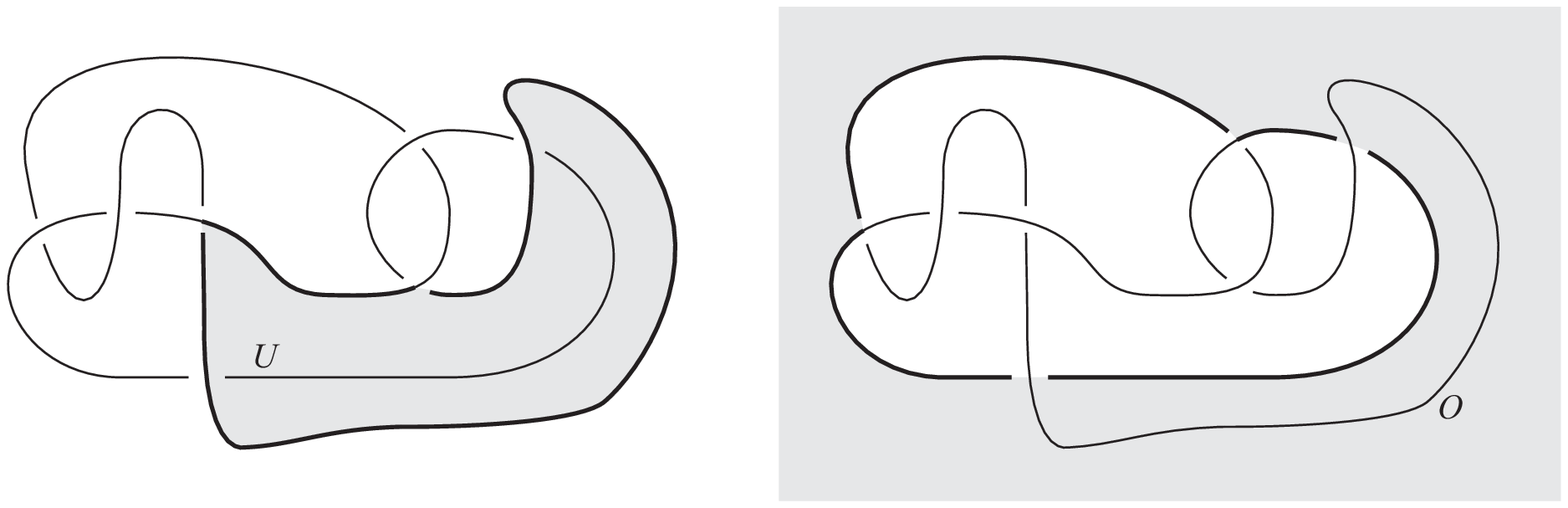}
    \end{center}
\mycap{Alternative $Z_2$ moves one could apply to the Culprit knot after the first $Z_3$.\label{culpritbis:fig}}
\end{figure}
\item The region $\Omega$ to which the final $Z_1$ move is applied is unbounded.
\end{itemize}

\paragraph{The Goeritz knot}
The Goeritz knot of Fig.~\ref{goeritz:fig}-top/left is untangled as shown in the rest of the figure,
via four $Z_3$ moves (none of which horizontal) and one $Z_1$ (applied to an unbounded $\Omega$).
Here and elsewhere we consider a knot diagram to be untangled if it has fewer than three crossings, because any
diagram of a non-trivial knot has more than two.
\begin{figure}
    \begin{center}
    \includegraphics[scale=.6]{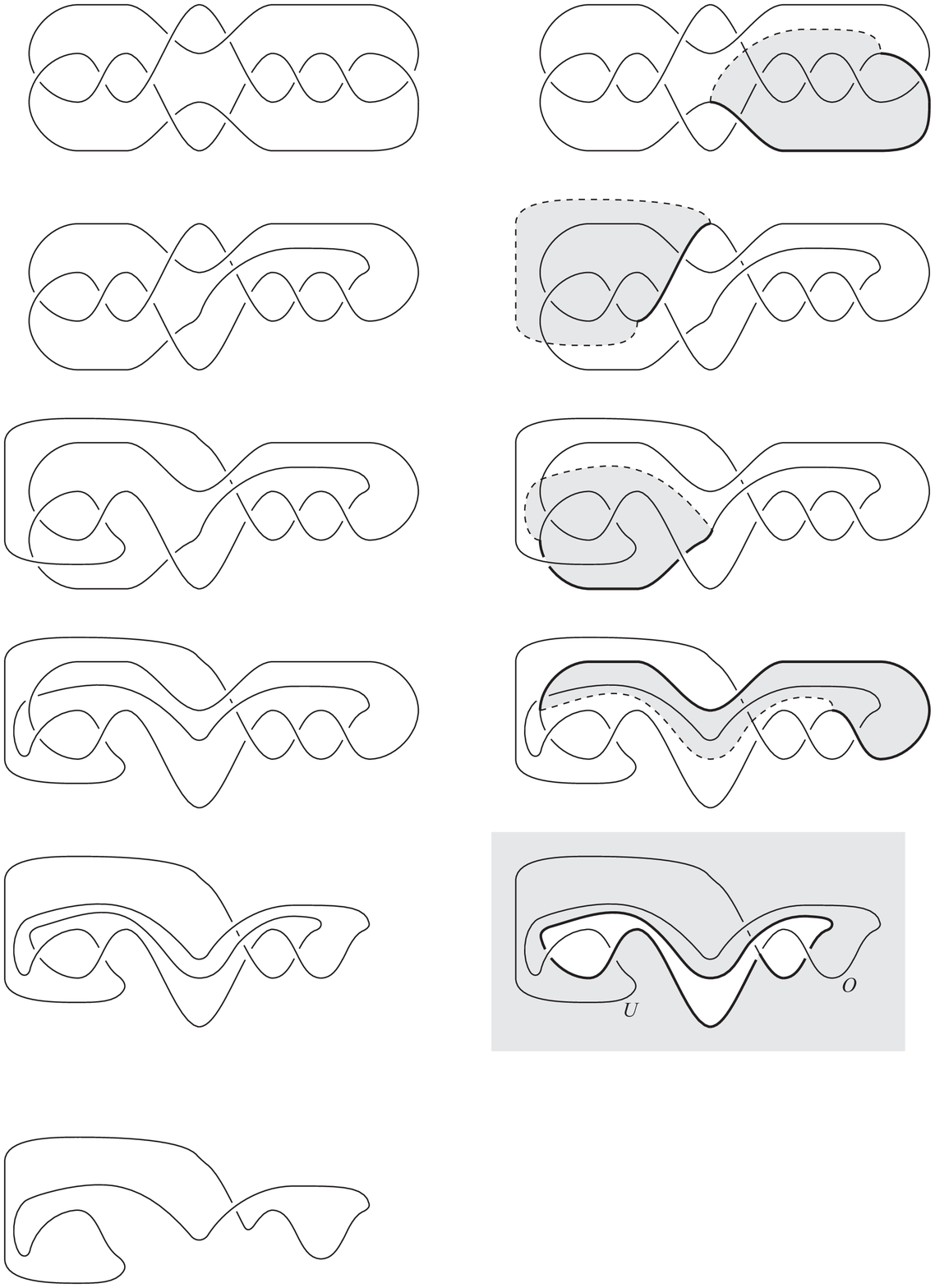}
    \end{center}
\mycap{Monotonic untangling of the Goeritz knot via four $Z_3$ and one $Z_1$ moves.\label{goeritz:fig}}
\end{figure}

\paragraph{The Thistlethwaite knot}
The Thistlethwaite knot of Fig.~\ref{thist:fig}-top/left
\begin{figure}
    \begin{center}
    \includegraphics[scale=.6]{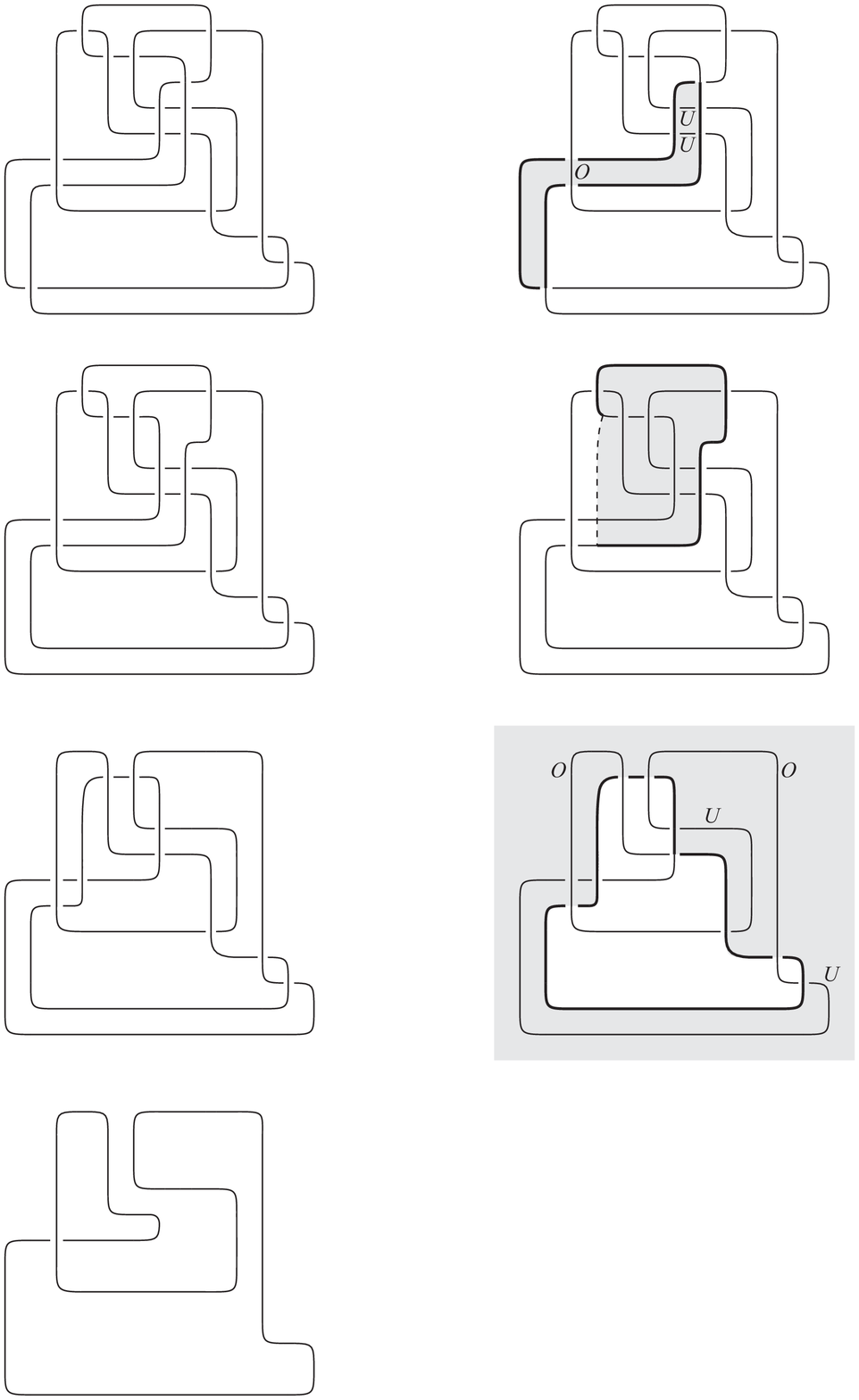}
    \end{center}
\mycap{Monotonic untangling of the Thistlethwaite knot via $Z_2$, $Z_3$ and $Z_1$.\label{thist:fig}}
\end{figure}
is untangled as shown in the rest of the figure,
via three $Z_{1,2,3}$ moves. Again all the moves strictly decrease the number of vertices, and
$Z_1$ is applied to an unbounded $\Omega$.

\paragraph{The Hass-Nowik knots}
The diagram of the unknot shown in Fig.~\ref{hassnowik:fig}-top has $7n-1$ crossings
\begin{figure}
    \begin{center}
    \includegraphics[scale=.6]{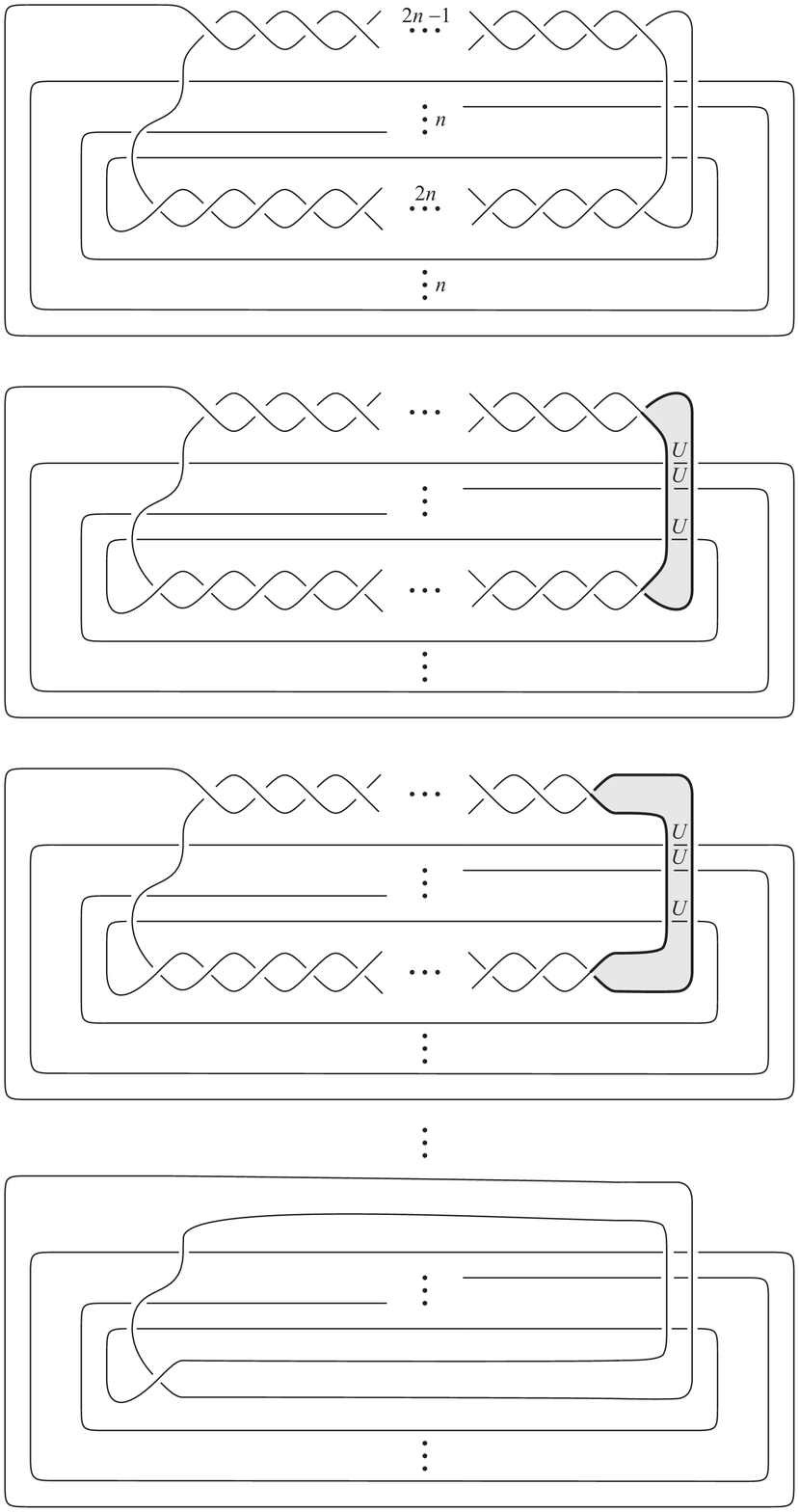}
    \end{center}
\mycap{A diagram of the unknot that untangles quadratically using the Reidemeister moves and linearly
using the moves $Z_{1,2,3}$: the initial sequence of $Z_2$ moves.\label{hassnowik:fig}}
\end{figure}
and was shown in~\cite{HaNo} to require at least $2n^2+3n-2$ Reidemeister moves to
reduce to the trivial diagram. In the rest of Fig.~\ref{hassnowik:fig} we show how to apply
$2n-1$ moves of type $Z_2$, thus reducing to $3n+1$ crossings, and then in Fig.~\ref{hassnowikbis:fig}
\begin{figure}
    \begin{center}
    \includegraphics[scale=.6]{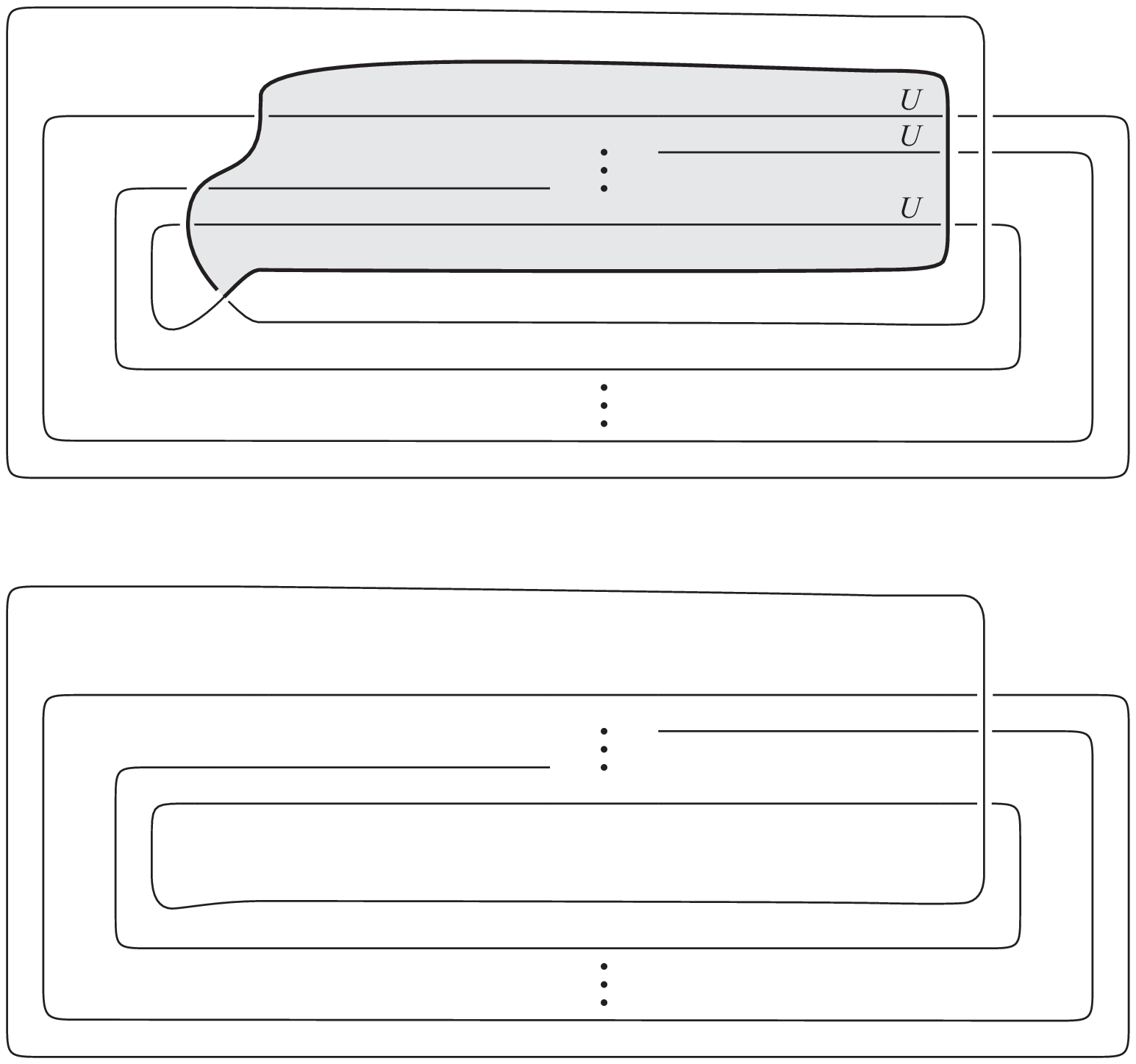}
    \end{center}
\mycap{A $Z_1$ move on the last diagram of Fig.~\ref{hassnowik:fig}, after which a sequence of $R_1$ moves
gives the trivial diagram.\label{hassnowikbis:fig}}
\end{figure}
we show a move $Z_1$ that removes $2n+1$ crossings, after which an obvious sequence of $n$ moves $R_1$ untangles the
diagram. This shows that while $2n^2+3n-2$ Reidemeister moves are necessary, $3n$ (strictly decreasing)
moves $Z_{1,2,3}$ suffice.

\paragraph{Necessity of horizontal moves}
In Figure.~\ref{K31:fig}-top/left we show a certain diagram $K_{31}$ of the unknot,
also considered by Dynnikov~\cite{dynbis}.
\begin{figure}
    \begin{center}
    \includegraphics[scale=.6]{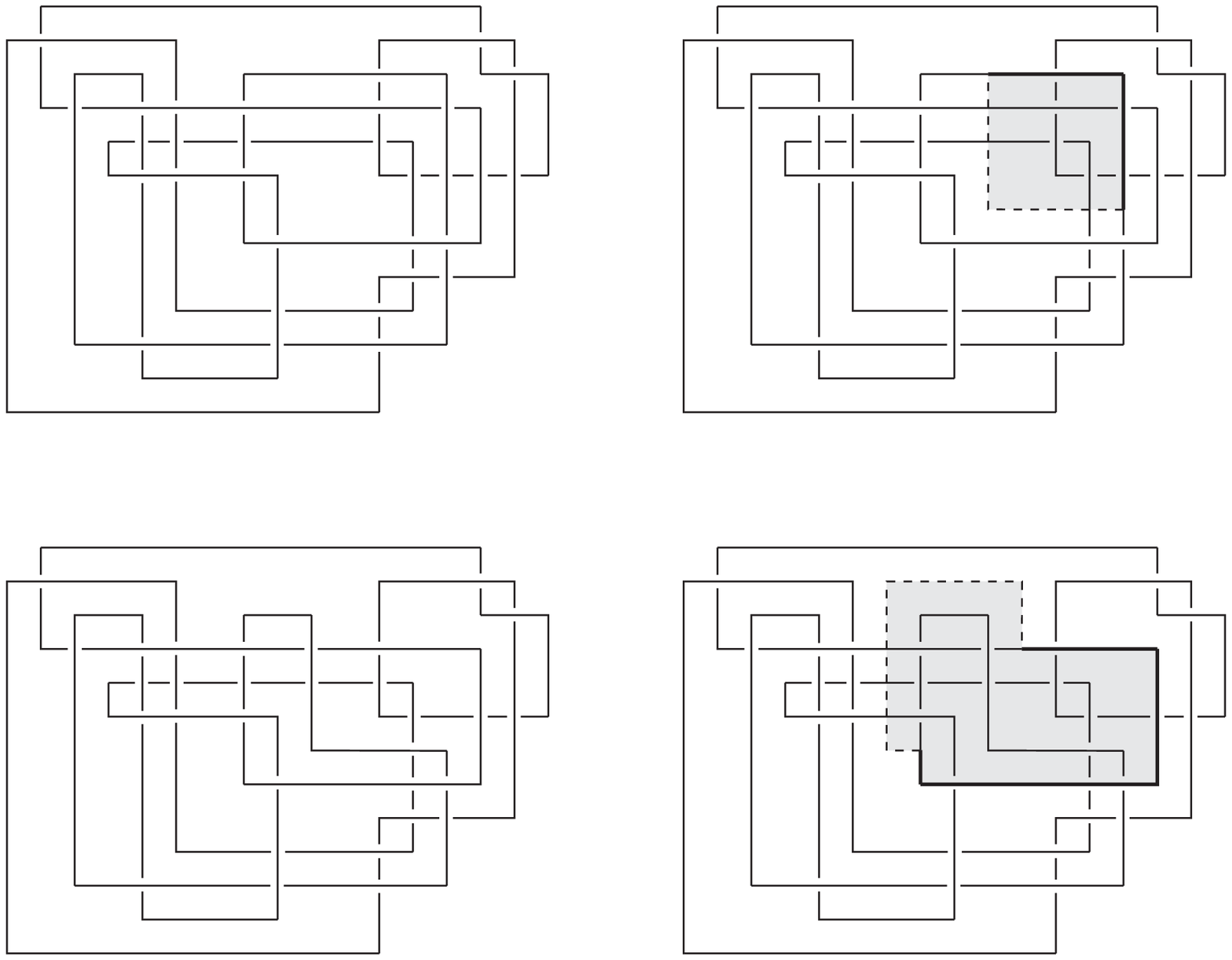}
    \end{center}
\mycap{A diagram of the unknot that untangles in a monotonic but not in a strictly monotonic way (first two of a sequence
of $10$ untangling $Z_{1,2,3}$ moves).\label{K31:fig}}
\end{figure}
Direct inspection shows (and our software~\cite{zanweb} confirms) that no decreasing $Z_{1,2,3}$ move applies to this diagram.
However, as shown in the same picture, after a horizontal $Z_3$ move a decreasing move is possible.
As illustrated in~\cite{zanweb}, a sequence of $8$ further $Z_{1,2,3}$ moves (all decreasing)
leads to a complete untanglement of the diagram.

\paragraph{More hard diagrams of the unknot}
The reader will find in~\cite{zanweb} full details on the following examples, all of which
refer to diagrams of the unknot that do not monotonically untangle via Reidemeister moves:
\begin{itemize}

\item A diagram apparently due to Kauffman, with $9$ crossings, monotonically untangled by $\calP$ via $3$ moves;

\item The ``monster'' diagram (a name apparently also
due to Kauffman), with 10 crossings, that untangles via $3$ moves;

\item A certain $K_{12}$ diagram, with 12 crossings, that untangles via $4$ moves;

\item The two Ochiai diagrams, with $13$ and $16$ crossings respectively, that untangle
via 5 and 6 moves;

\item A diagram with $32$ crossings, apparently due to Freedman, that untangles via $8$ moves;

\item A satellite diagram of the unknot with 64 crossings, apparently due to Hass,
monotonically untangled by procedure $\calP$ via $31$ moves;

\item The \emph{Haken Gordian knot}, with $141$ crossings,
untangled via $53$ moves;

\item The \emph{Haken satellite knot}, with $188$ crossings,
untangled via $68$ moves.

\end{itemize}

\paragraph{A big unknot}
We show in Fig.~\ref{fulltwist:fig} a diagram of the unknot provided to us by the referee
(in fact a variation of the diagram of Fig.~\ref{120:fig} below).
\begin{figure}
    \begin{center}
    \includegraphics[scale=.6]{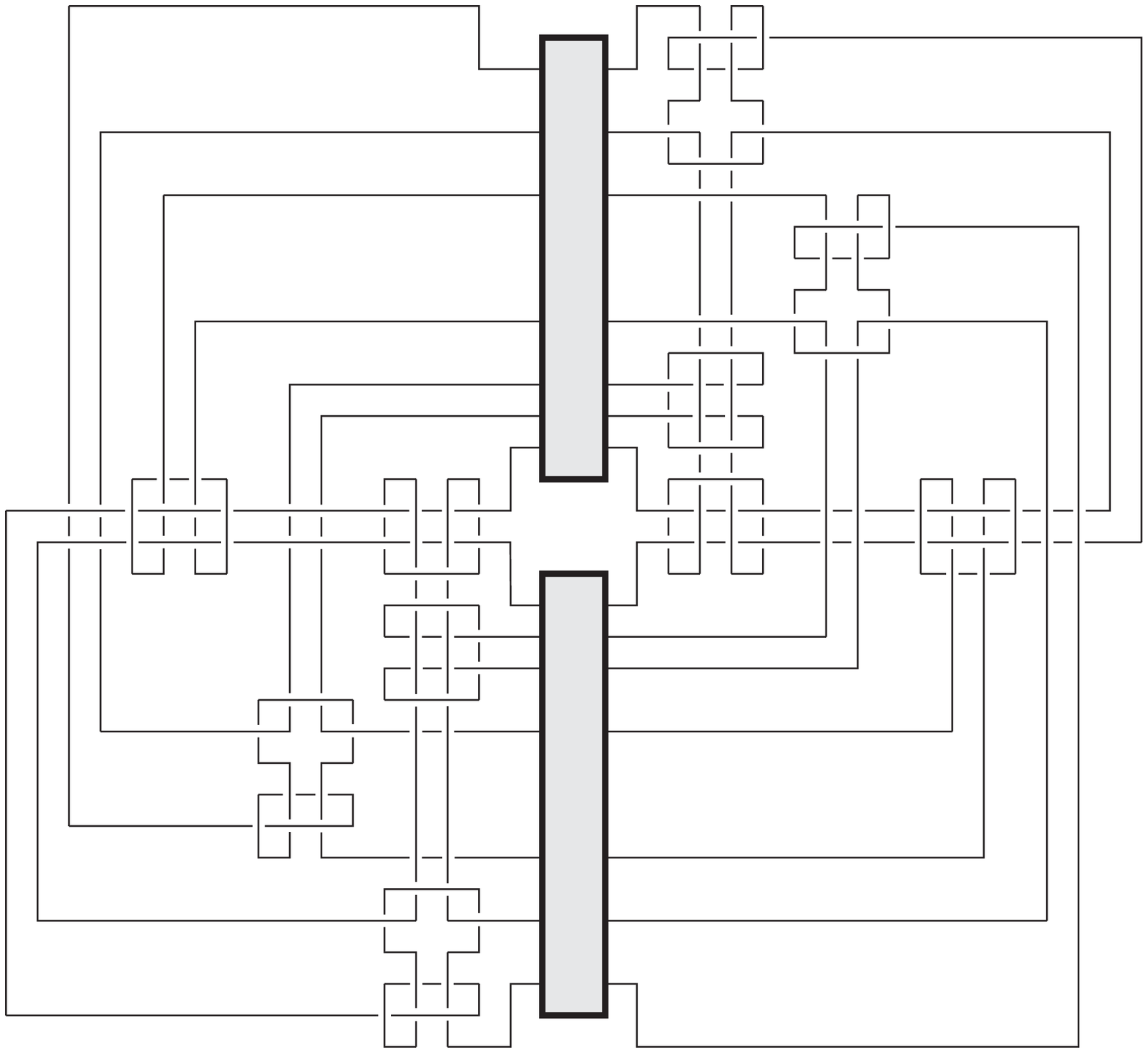}
    \end{center}
\mycap{By inserting a positive full twist of the seven strands entering the shadowed
lower box in the picture, and a negative full twist of the seven strands entering the shadowed
upper box in the picture, we get a diagram of the unknot with 204 crossings.\label{fulltwist:fig}}
\end{figure}
Our reduced procedure untangles it via a strictly monotonic sequence of $80$ moves of type $Z_{1,2,3}$.

\paragraph{Multiple horizontal moves}
For all the diagrams mentioned so far, the simplification was achieved
through a sequence of non-increasing $Z_{1,2,3}$ moves with \emph{at most one} consecutive
horizontal $Z_3$. We describe here a first examples showing that multiple horizontal $Z_3$'s may be required.
It was provided to us by the referee, and it is shown in
Figure~\ref{138:fig}.
\begin{figure}
    \begin{center}
    \includegraphics[scale=.6]{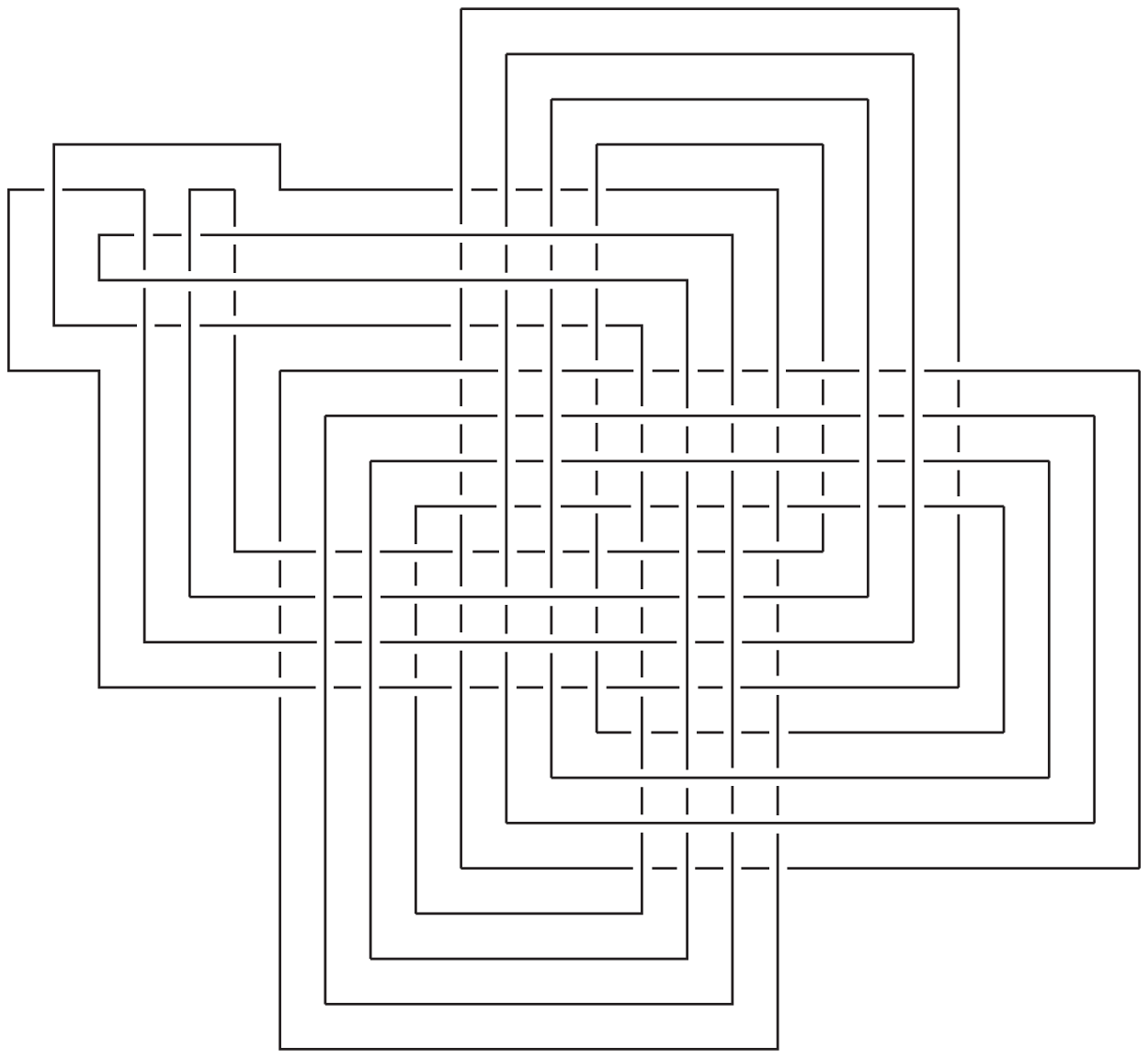}
    \end{center}
\mycap{A diagram of the unknot with 138 crossings.\label{138:fig}}
\end{figure}
Our reduced algorithm manages to untangle it using a sequence of $37$ moves of
type $Z_{1,2,3}$, three of which horizontal $Z_3$'s, including two consecutive ones at level 137.
See also Fig.~\ref{138C:fig} below.

\paragraph{Quick simplifications}
We mention here two more diagrams of the unknot, both with 7 crossings,
one from~\cite{PraSos} and one known as the ``nasty'' diagram. They both
simplify monotonically also via Reidemeister moves, but their untanglement
via moves $Z_{1,2,3}$ is particularly efficient: only one $Z_1$ suffices for the first diagram,
while a $Z_2$ and a $Z_1$ suffice for the second one.

\paragraph{Simplification of diagrams of non-trivial knots}
In Fig.~\ref{trefoil:fig}-left
\begin{figure}
    \begin{center}
    \includegraphics[scale=.6]{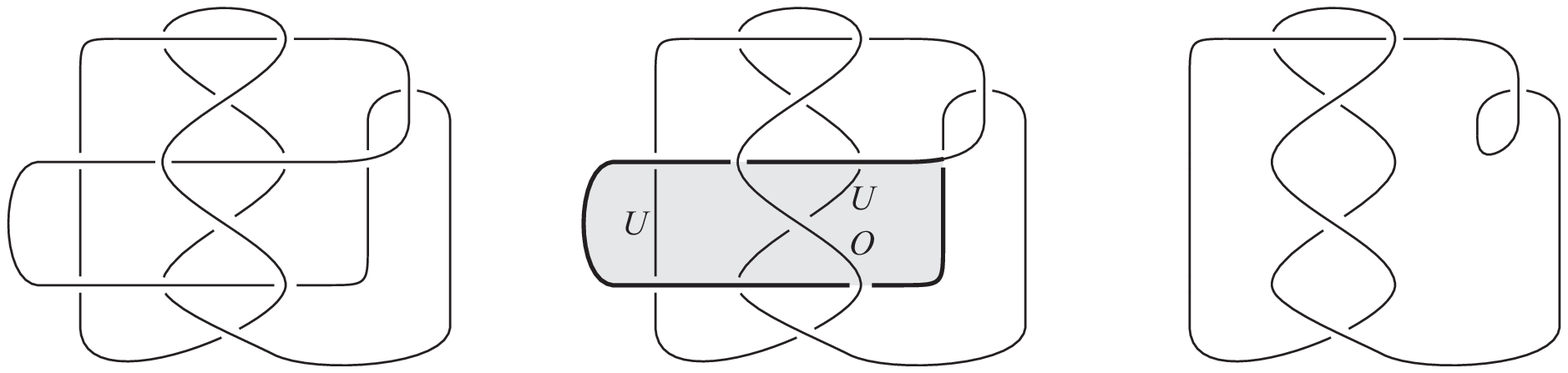}
    \end{center}
\mycap{Simplification of a diagram of the trefoil.\label{trefoil:fig}}
\end{figure}
we show a diagram~\cite{PraSos} of the trefoil, and the rest of the picture proves that
it reduces to the usual minimal diagram via one $Z_1$ move, followed by an obvious
$R_1$ and an obvious $R_2$ (not illustrated). We should mention that this diagram is also
monotonically reduced to a minimal one using Reidemeister moves only, but six of them are
required.

Something similar happens
for the diagram from~\cite{PraSos} of the figure-eight knot,
shown in Fig.~\ref{eight:fig}.
\begin{figure}
    \begin{center}
    \includegraphics[scale=.6]{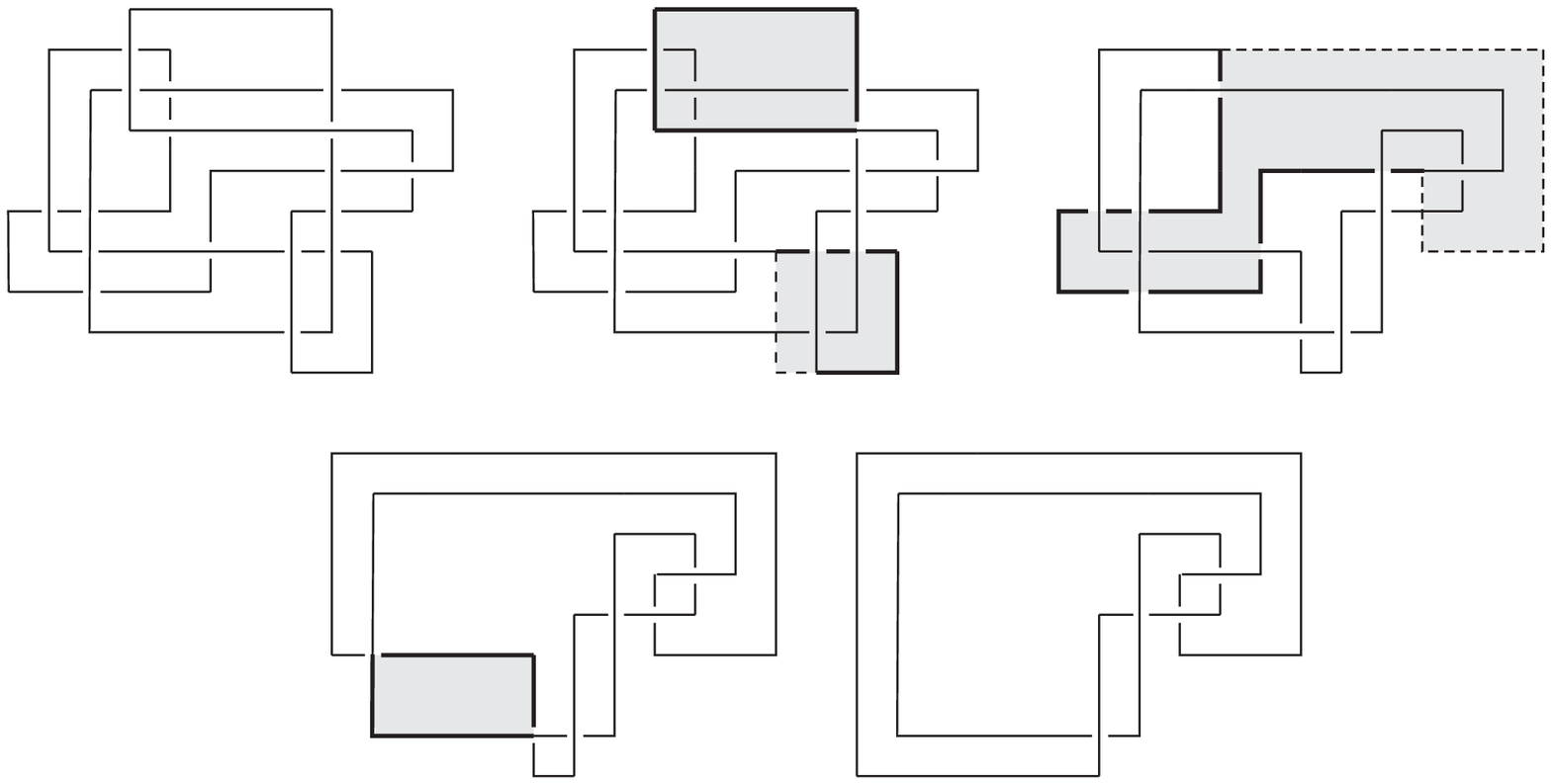}
    \end{center}
\mycap{Simplification of a diagram of the figure-eight knot using a move $Z_1$ and
a simultaneous $Z_3$, then a $Z_3$ and lastly a $Z_2$ (actually, an $R_2$).\label{eight:fig}}
\end{figure}
Via Reidemeister moves, this diagram reduces to 4 crossings monotonically, but only in a dozen passages,
while procedure $\calP$ completely simplifies it in 4 moves (in various different ways).
Note that in Fig.~\ref{eight:fig} we show on the same diagram the result of a move and the identification of the next one, and we will do the same henceforth.

\paragraph{The Kazantsev knot}
A rather complicated non-trivial knot we treat comes from~\cite{kaz} and is shown in Fig.~\ref{kazanknot:fig}.
\begin{figure}
    \begin{center}
    \includegraphics[scale=.6]{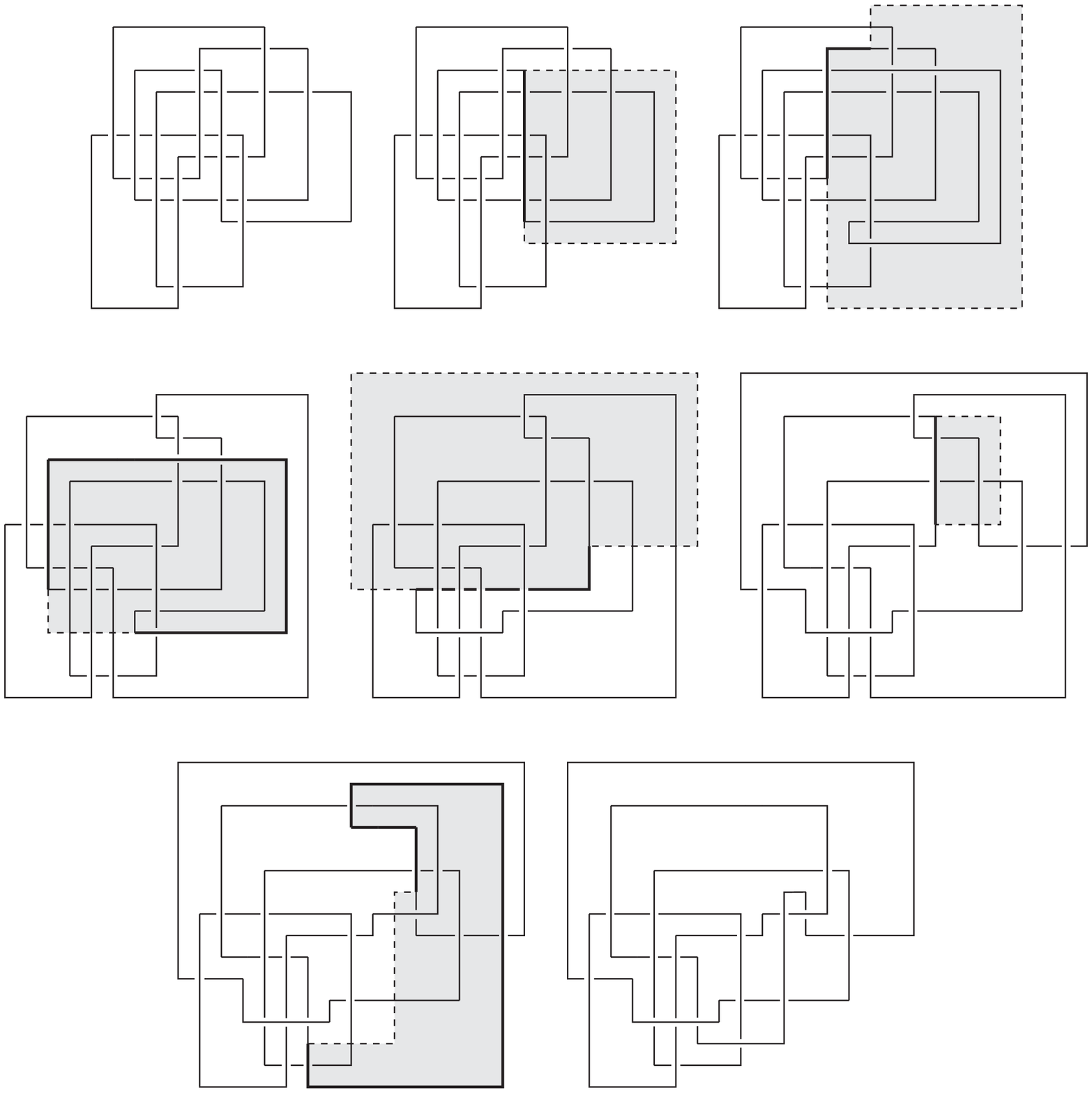}
    \end{center}
\mycap{(Non-strictly) monotonic simplification via $Z_3$ of the Kazantsev knot.\label{kazanknot:fig}}
\end{figure}
This diagram has 23 crossings and cannot be monotonically simplified via Reidemeister moves,
whereas an application of procedure $\calP$ with 6 moves of type $Z_3$ (one of which is horizontal, and actually an $R_3$)
reduces it to its minimal status with $17$ crossings.

\paragraph{A composite knot}
Another interesting example of the efficiency of procedure $\calP$ is illustrated in Fig.~\ref{trefeight:fig}.
\begin{figure}
    \begin{center}
    \includegraphics[scale=.6]{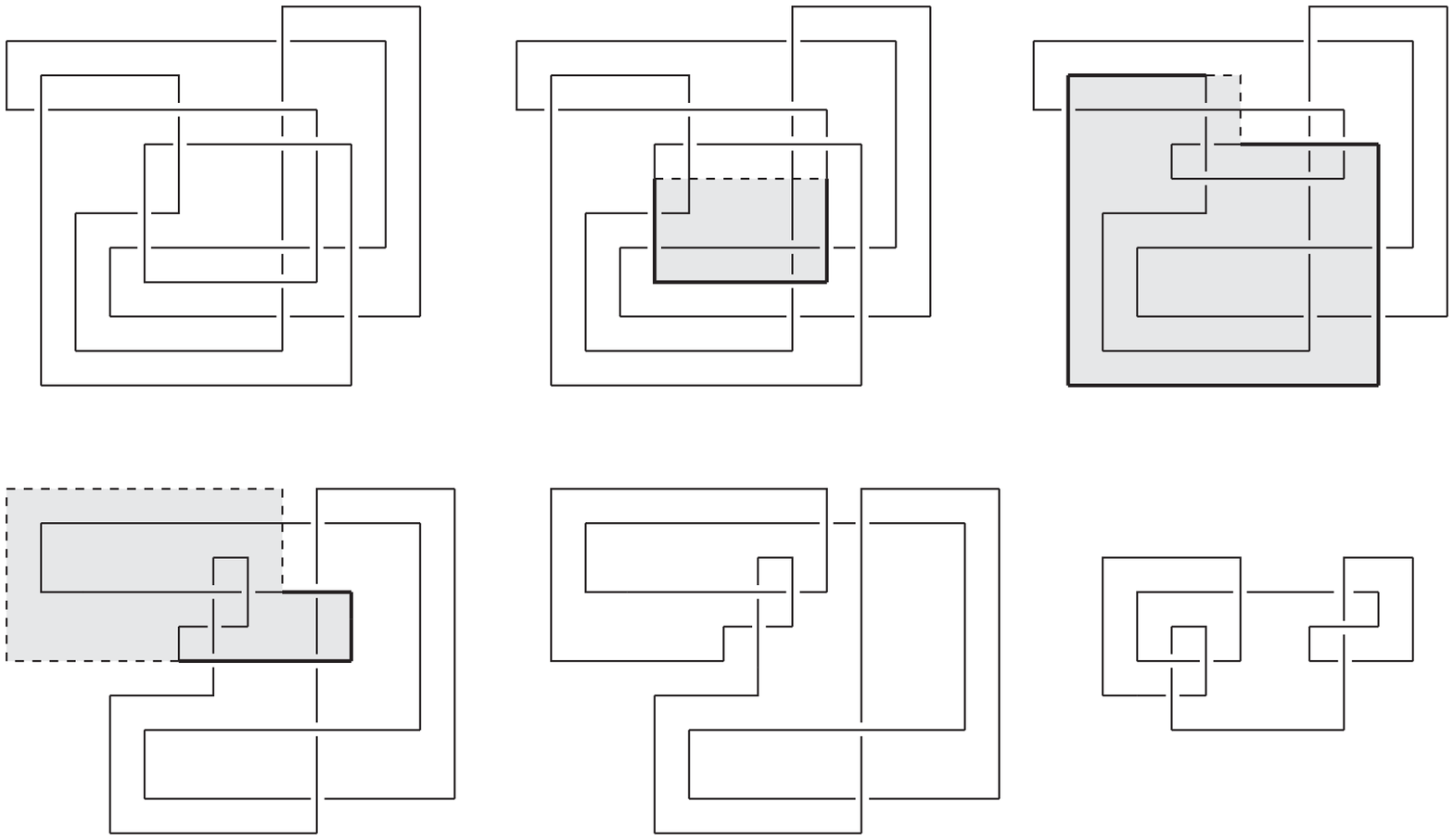}
    \end{center}
\mycap{Simplification of a connected sum of a trefoil and a figure-eight knot.\label{trefeight:fig}}
\end{figure}
Here we show a diagram with 15 crossings that cannot be monotonically
simplified via Reidemeister moves, while three strictly decreasing $Z_3$'s
lead to a minimal diagram,
that actually identifies the
knot as the connected sum of a trefoil and a figure-eight knot (the last diagram is a mere redrawing
of the previous one, to make the connected sum structure more apparent).

\paragraph{Two big non-trivial knots}
We conclude by showing two diagrams of non-trivial knots that the reduced procedure $\calP$ brings to
a minimal status but only using more than one consecutive horizontal $Z_3$.
The first diagram was suggested to us by Malik Obeidin, and it is the diagram shown in
Fig.~\ref{malik:fig}.
\begin{figure}
    \begin{center}
    \includegraphics[scale=.6]{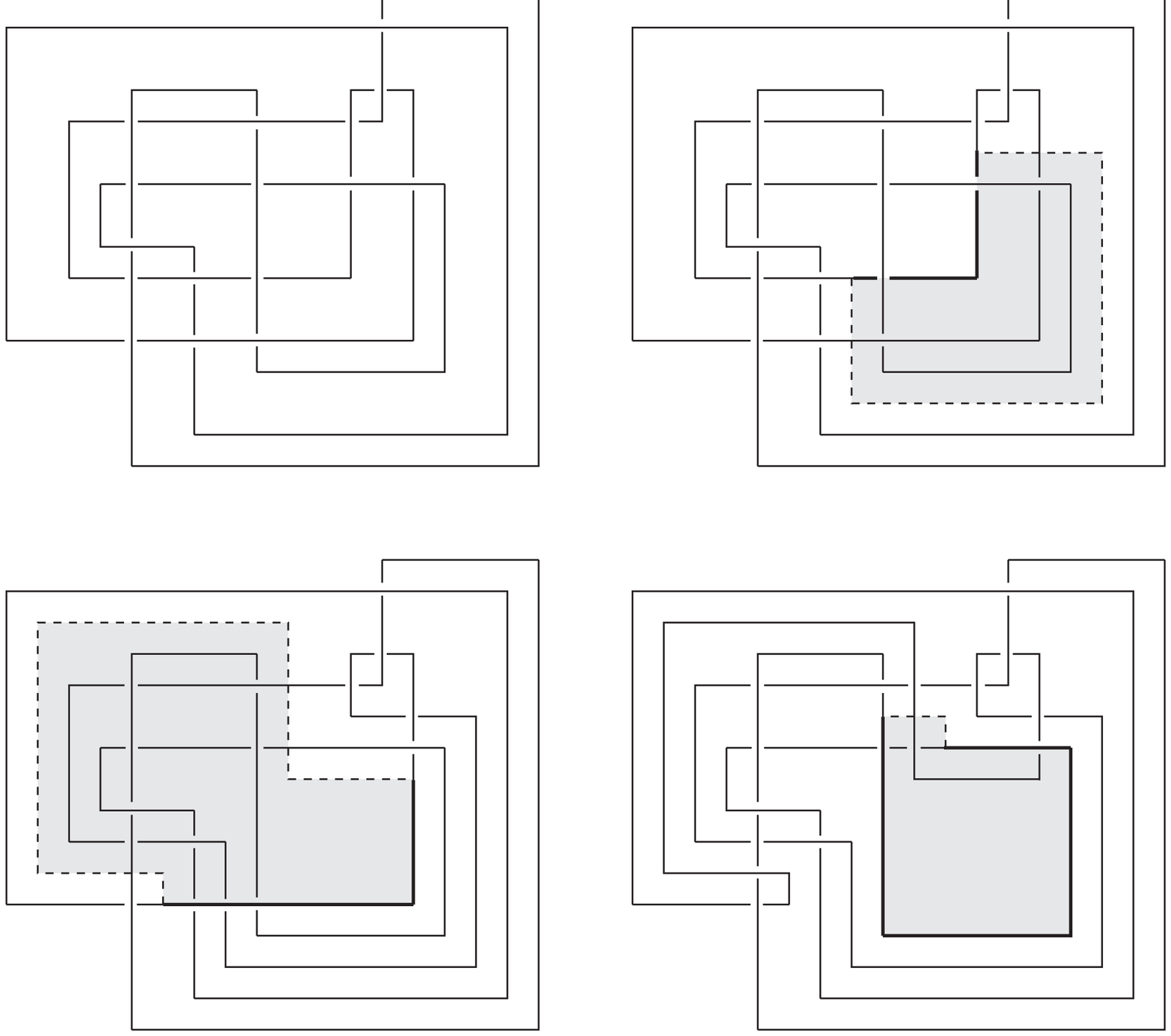}
    \end{center}
\mycap{A 16-crossing diagram and two horizontal $Z_3$ moves applied to it,
after which there exists a $Z_3$ decreasing $c$ by $1$ (outcome of the move not shown).\label{malik:fig}}
\end{figure}
The same figure shows how to transform it into a minimal diagram
via two horizontal $Z_3$ moves followed by a decreasing one.
The second diagram may be found in~\cite{zanweb}. It has
$32$ crossings and $\calP$ turns it into a minimal diagram with $28$
via a sequence of 10 moves, and more precisely:
\begin{itemize}
\item Two horizontal $Z_3$ moves at level $32$;
\item One decreasing $Z_2$ move leading to level $30$;
\item Three horizontal $Z_2$ moves at level $30$;
\item One decreasing move $Z_3$ leading to level $29$;
\item Two horizontal moves $Z_3$ at level $29$;
\item One decreasing move $Z_3$ leading to level $28$.
\end{itemize}

\section{Detection of composite links\\ and the complete algorithm}\label{Cmove:section}

In this section we introduce a further move $C$ (and a variant $\Ctil$ of it)
that applied to a diagram $D$ of a link $L$ gives diagrams $D_0,D_1$ of links $L_0,L_1$
such that $L$ is the connected sum of $L_0$ and $L_1$.  Moreover $D_0$ and $D_1$ have
strictly fewer crossings than $D$.  This allows to break the task of simplifying $D$ into
the same task separately for $D_0$ and $D_1$.  In particular:
\begin{itemize}
\item If via moves $Z_{1,2,3}$ we can show that
both $[D_0]$ and $[D_1]$ are trivial, we can conclude that $[D]$ is also trivial;
\item If via moves $Z_{1,2,3}$ we can show that one of $D_0$ or $D_1$ is trivial
and we can transform the other one into a minimal $D'$, we can conclude that
$c([D])=c(D')$;
\item If via moves $Z_{1,2,3}$ we can transform $D_0$ or $D_1$ into minimal
$D_0'$ and $D_1'$, \emph{assuming that the conjectural~\cite{lack1}
additivity of the crossing number is true},
we can conclude that $c([D])=c(D_0')+c(D_1')$.
\end{itemize}
Our software~\cite{zanweb} already contains the implementation of the moves $C$ and $\Ctil$
and performs the simplification of diagrams using them (but we recall that it
currently handles knots only). The software can be asked by
the user to look for a move $C$ or $\Ctil$, while
in the quick simplification procedure it only uses these moves if it gets stuck with the $Z_{1,2,3}$.

\paragraph{The move $C$}
Let $p:\mathop{\sqcup}\limits_{i=1}^p S_i^1\to S^2$ be the immersion associated to a link diagram $D$.
Let $\Omega\subset S^2$ be a tame topological disc with $\partial\Omega$ transverse to $D$,
and $\alpha,\beta_1,\ldots,\beta_N$ be the components of
$p^{-1}(\Omega)$. Suppose that $\alpha$ is an arc and that one can assign labels
$\lambda_1,\ldots,\lambda_N$ in $\{U,O\}$ to $\beta_1,\ldots,\beta_N$ so that
(see Fig.~\ref{Cmove:fig}):
\begin{figure}
    \begin{center}
    \includegraphics[scale=.45]{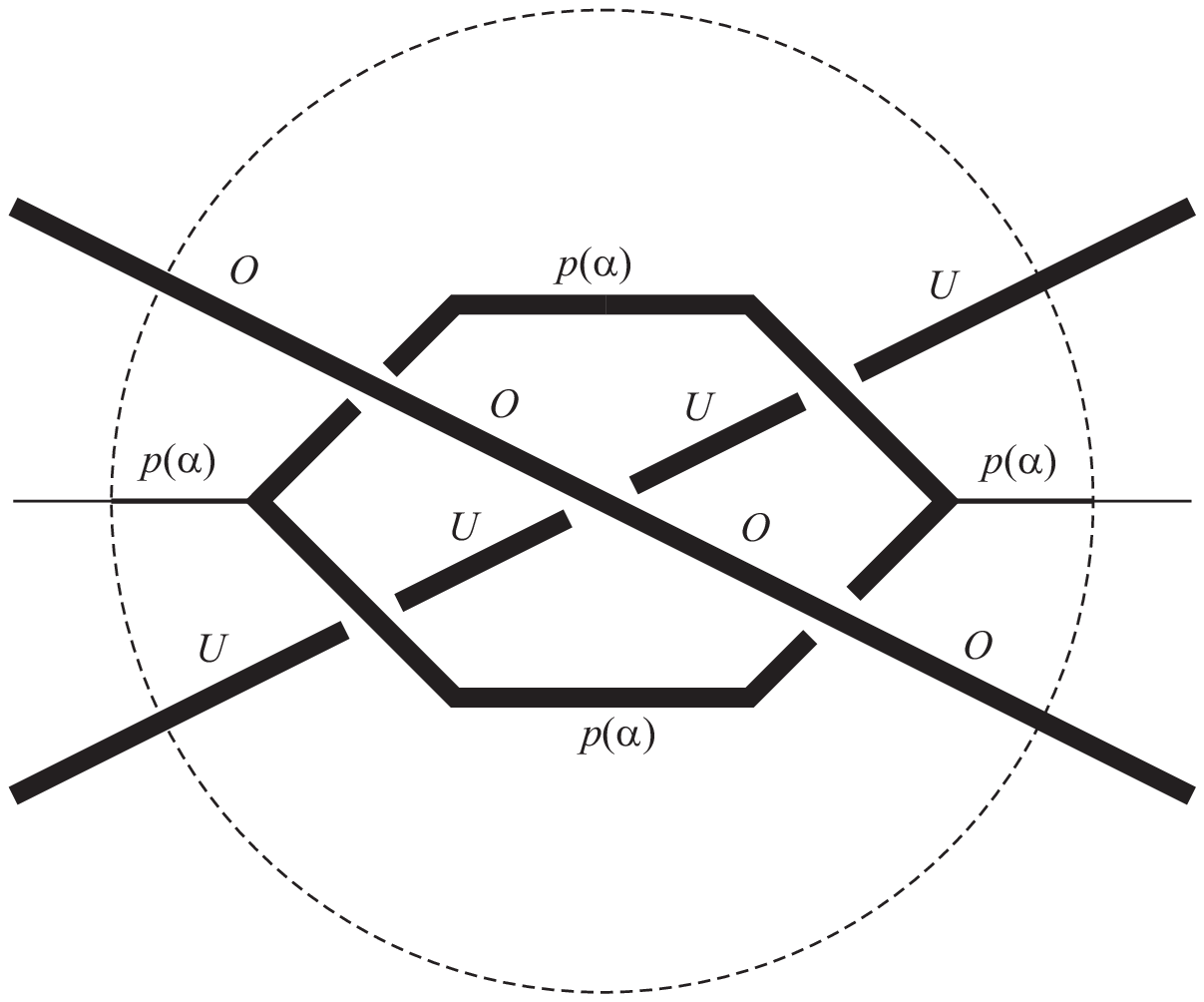}
    \end{center}
\mycap{A diagram $D$ to which a move $C$ applies.\label{Cmove:fig}}
\end{figure}
\begin{itemize}
\item If $\lambda_j=O$ then $p(\beta_j)$ is over $p(\alpha)$ where they cross;
\item If $\lambda_j=U$ then $p(\beta_j)$ is under $p(\alpha)$ where they cross;
\item If $\lambda_j=O$ and $\lambda_k=U$ then $p(\beta_j)$ is over $p(\beta_k)$ where they cross.
\end{itemize}
Let $\gamma$ be any of the two halves of $\partial\Omega$ into
which the ends of $p(\alpha)$ split it.
Given these data, we call $C$ the move that replaces $D$ by the pair of diagrams $(D_0,D_1)$, where (see Fig.~\ref{Cmove01:fig}):
\begin{itemize}
\item $D_0$ is $p(\alpha)$ union $\gamma$ with crossings as in $D$ (none on $\gamma$);
\item $D_1$ is $D$ minus $p(\alpha)$ with crossings as in $D$, union $\gamma$, with $\gamma$ under any
$\beta_j$ with $\lambda_j=O$, and over any $\beta_j$ with $\lambda_j=U$.
\end{itemize}
\begin{figure}
    \begin{center}
    \includegraphics[scale=.45]{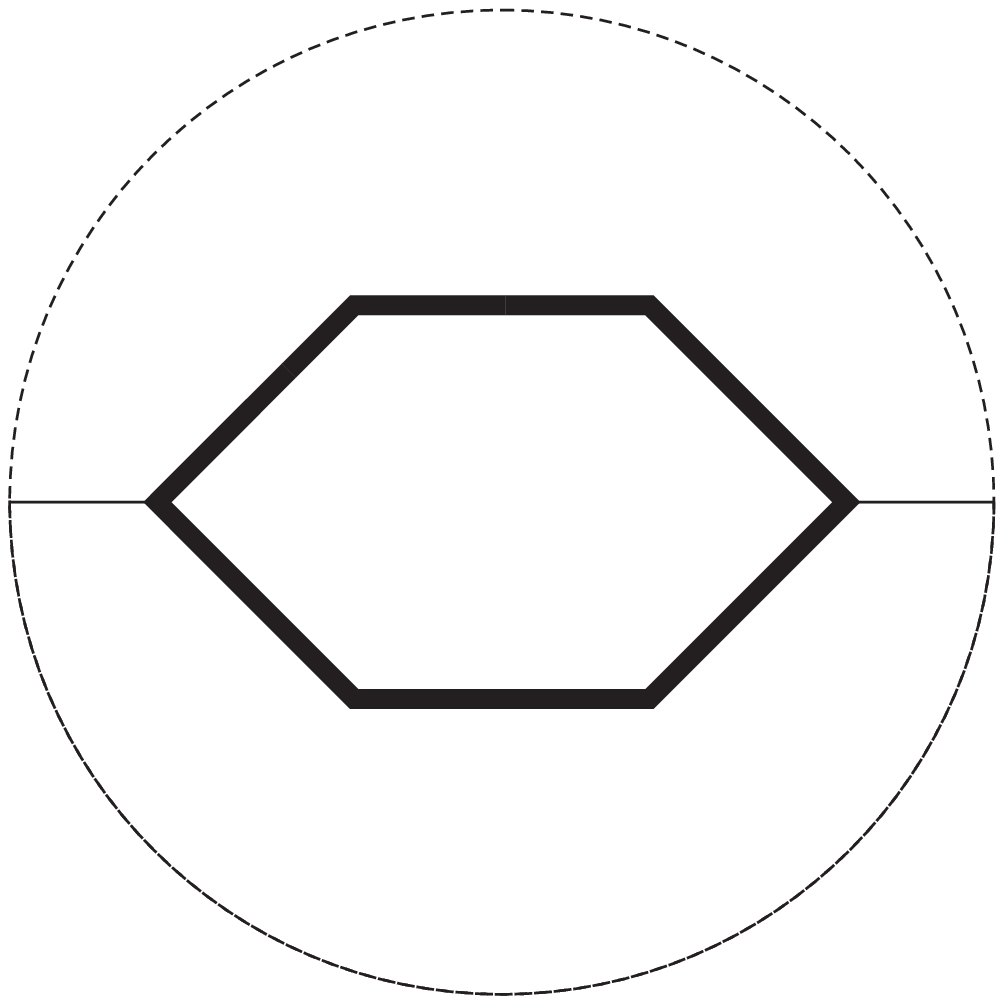}\qquad
    \includegraphics[scale=.45]{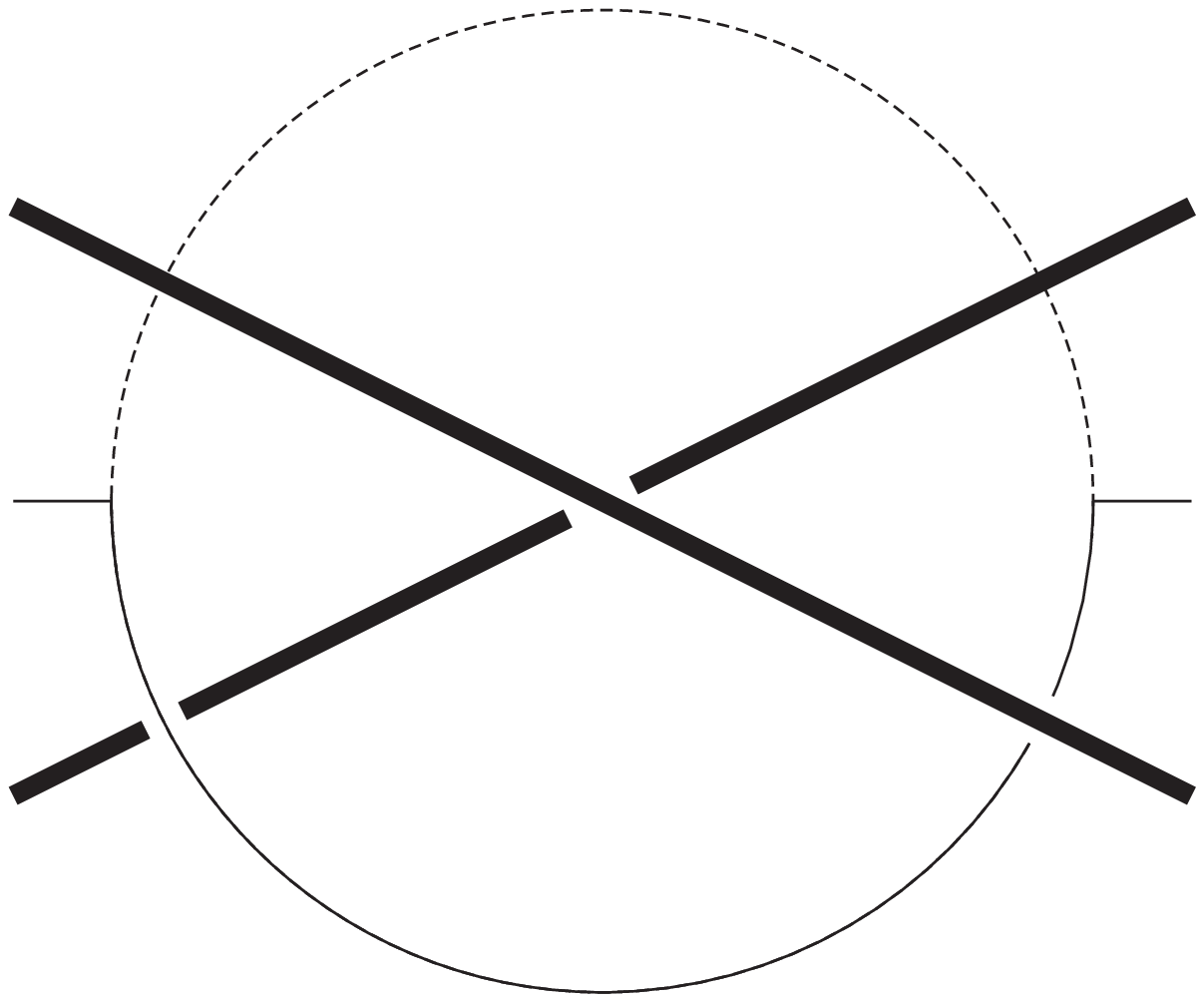}
    \end{center}
\mycap{The diagrams $D_0$ and $D_1$ produced by $C$.\label{Cmove01:fig}}
\end{figure}

\begin{prop}
$[D]=[D_0]\#[D_1]$.
\end{prop}

\begin{proof}
The assumptions imply that $[D]$ can be realized as a link $L$ in $\matR^3$ such that
$L\cap(\Omega\times\matR)$ is the disjoint union of three portions:
\begin{itemize}
\item One in $\Omega\times(-\varepsilon,\varepsilon)$, with projection $p(\alpha)$;
\item One in $\Omega\times(1-\varepsilon,1+\varepsilon)$, with projection $\bigcup\{p(\beta_i):\ \lambda_i=O\}$;
\item One in $\Omega\times(-1-\varepsilon,-1+\varepsilon)$, with projection $\bigcup\{p(\beta_i):\ \lambda_i=U\}$.
\end{itemize}
So the sphere $\partial(\Omega\times[-\varepsilon,\varepsilon])$ meets $L$
transversely at two points, whence it allows to express $L$ as $L_0\# L_1$, and of course $L_j=[D_j]$.
\end{proof}

\begin{rem}
\emph{If in the definition of the move $C$ we have that $\alpha$ is not an arc but
one of the circles $S^1_i$ then
$L$ is the split link $[D_0]\sqcup[D_1]$, where $D_0$ is $p(\alpha)$
and $D_1$ is $D$ with $p(\alpha)$ removed, both with crossings as in $D$.}
\end{rem}

We now note the following:
\begin{itemize}
\item If the ends of $p(\beta_i)$ separate the ends of $p(\alpha)$ on $\partial\Omega$
then $p(\beta_i)$ crosses $p(\alpha)$ at least once.
\end{itemize}
We also stipulate that the move $C$ can only be applied if the following
further assumption is met (otherwise we can replace $\Omega$ by a smaller disc):
\begin{itemize}
\item If the ends of $p(\beta_i)$ do not separate the ends of $p(\alpha)$ on $\partial\Omega$
then $p(\beta_i)$ crosses $p(\alpha)$ at least twice.
\end{itemize}
The remark and the assumption just made easily imply the following:

\begin{prop}\label{stupid:C:prop}
\begin{itemize}
\item $c(D_0)+c(D_1)\leqslant c(D)$;
\item $c(D_0)<c(D)$ unless $p^{-1}(\Omega)=\alpha$ and $D$ has no crossings outside $\Omega$, in which case $D_0=D$.
\item $c(D_1)<c(D)$ unless $p(\alpha)$ has no self-crossing and $p(\beta_j)$ crosses
$p(\alpha)$ as many times as it crosses $\gamma$, in which case $D_1=D$.
\end{itemize}
\end{prop}
Of course we do not apply the move $C$ in any of the situations described in the
last two items of this proposition.
An example of application of $C$ is provided in Figg.~\ref{Cexample:fig} and~\ref{Cexample01:fig}, and a
\begin{figure}
    \begin{center}
    \includegraphics[scale=.5]{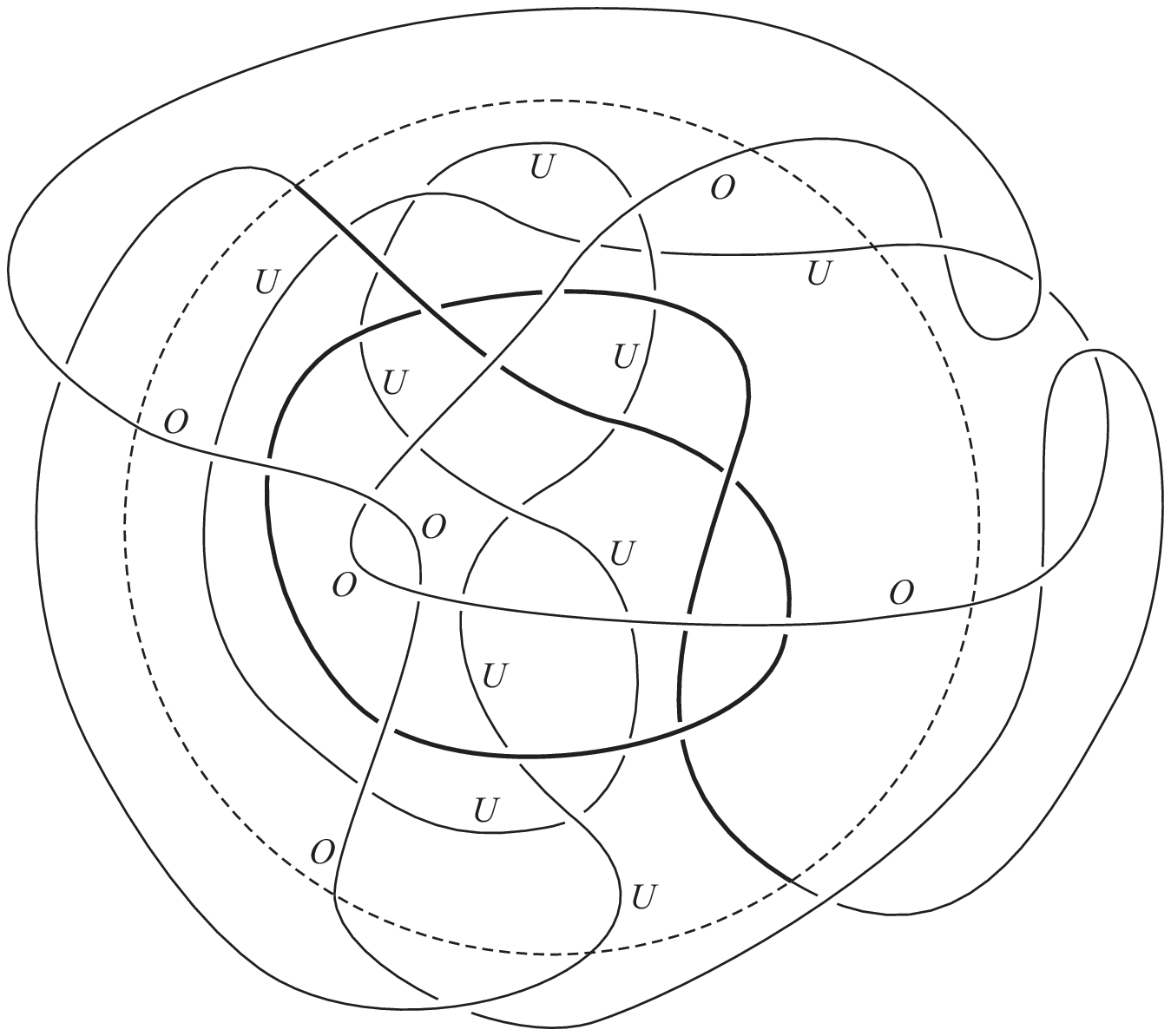}
    \end{center}
\mycap{A move $C$ applies to this diagram $D$. The disc $\Omega$ is the bounded region with the
dashed circle as boundary, and $p(\alpha)$ is the slightly thicker arc.\label{Cexample:fig}}
\end{figure}
\begin{figure}
    \begin{center}
    \includegraphics[scale=.5]{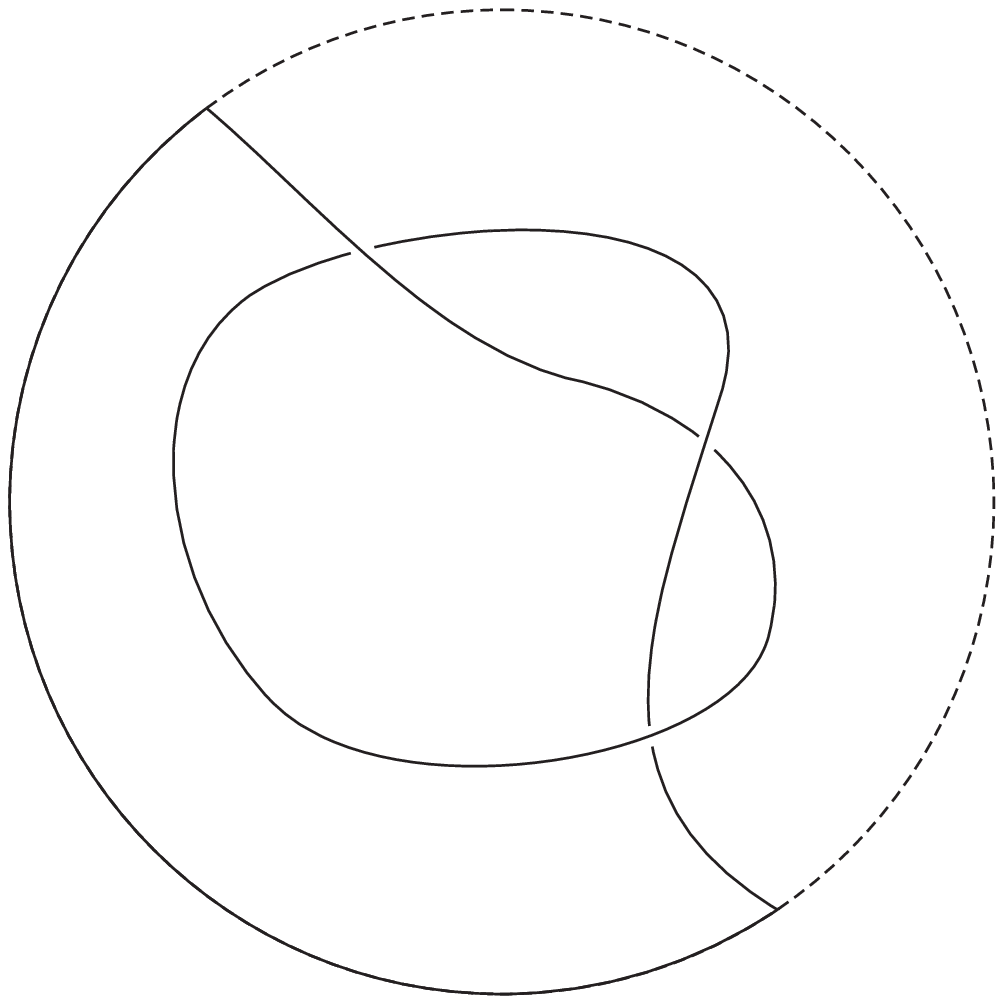}\qquad
    \includegraphics[scale=.5]{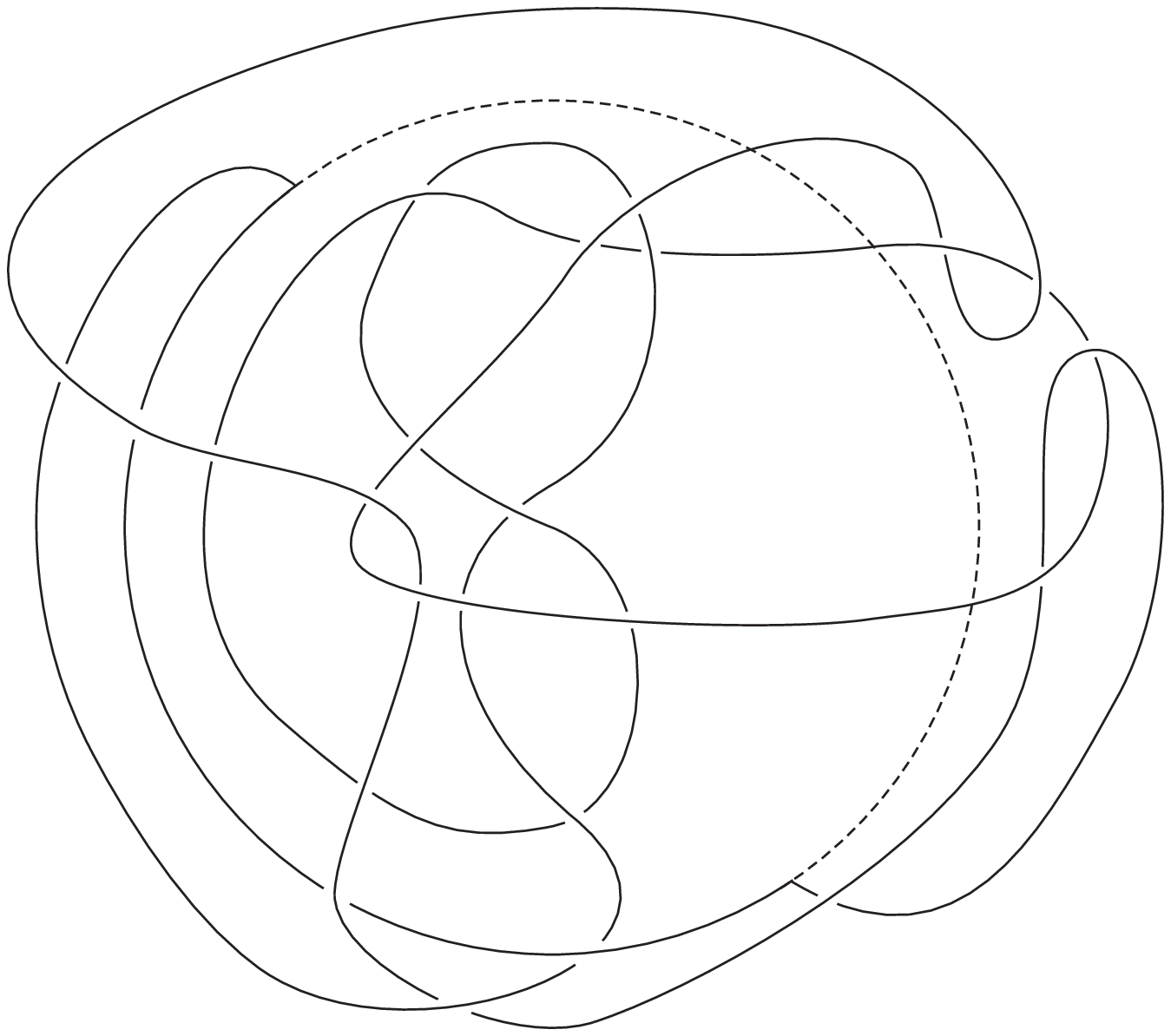}
    \end{center}
\mycap{Applyting $C$ to the $D$ of Fig.~\ref{Cexample:fig} we get these $D_0$ and $D_1$.\label{Cexample01:fig}}
\end{figure}
first example showing that $C$ is useful is described
in Fig.~\ref{138C:fig}, where
we prove that it applies to the 138-crossing diagram of the unknot
already considered in Fig.~\ref{138:fig}.
\begin{figure}
    \begin{center}
    \includegraphics[scale=.45]{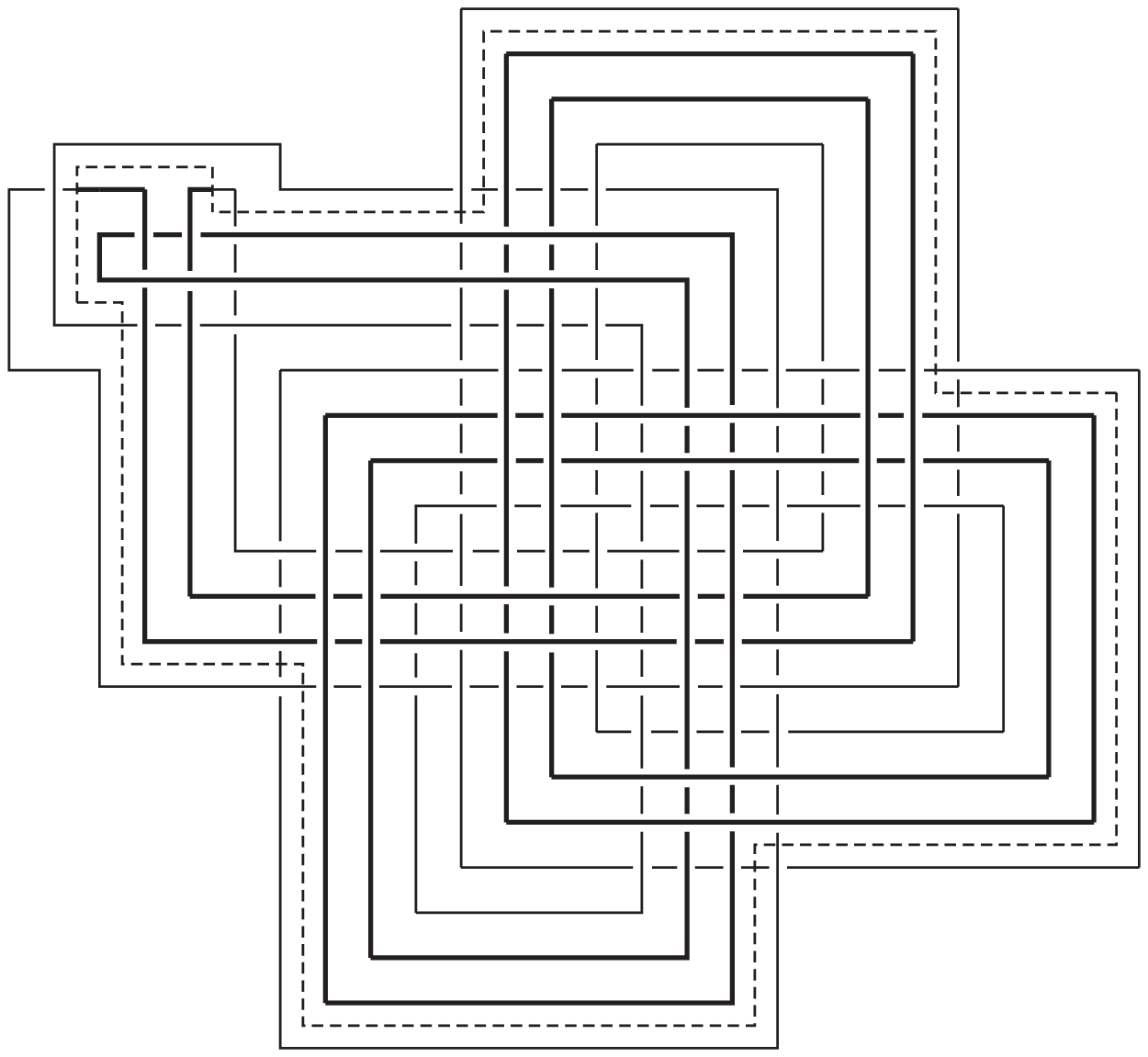}
    \end{center}
\mycap{A move $C$ applies to the diagram of Fig.~\ref{138:fig}.
The disc $\Omega$ is the bounded one with the dashed curve as boundary, $p(\alpha)$ is
the slightly thicker arc inside $\Omega$, and all the other arcs in $\Omega$ have label $U$.\label{138C:fig}}
\end{figure}
The two resulting diagrams, shown in Fig.~\ref{138C01:fig}
\begin{figure}
    \begin{center}
    \includegraphics[scale=.45]{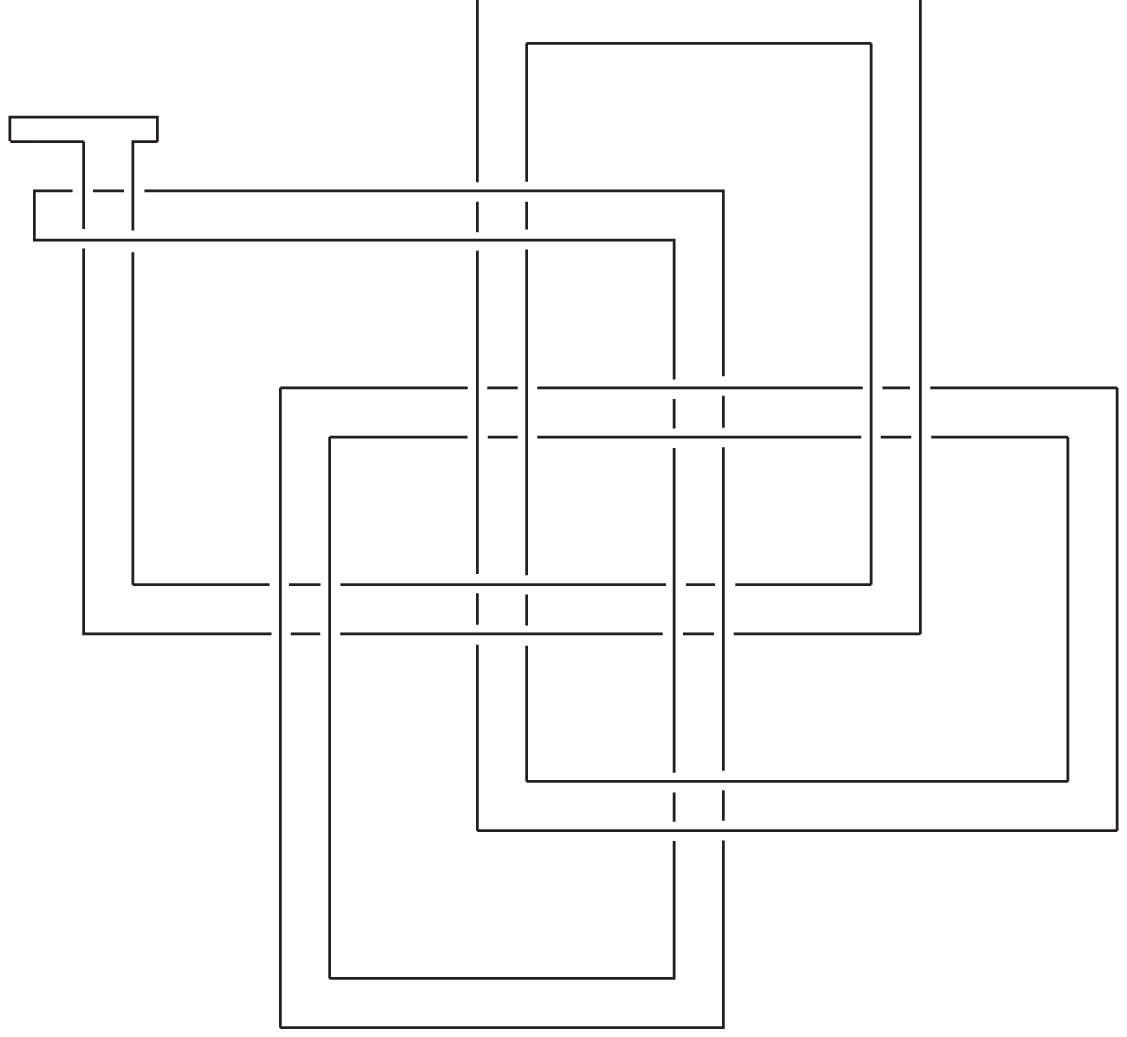}\qquad
    \includegraphics[scale=.45]{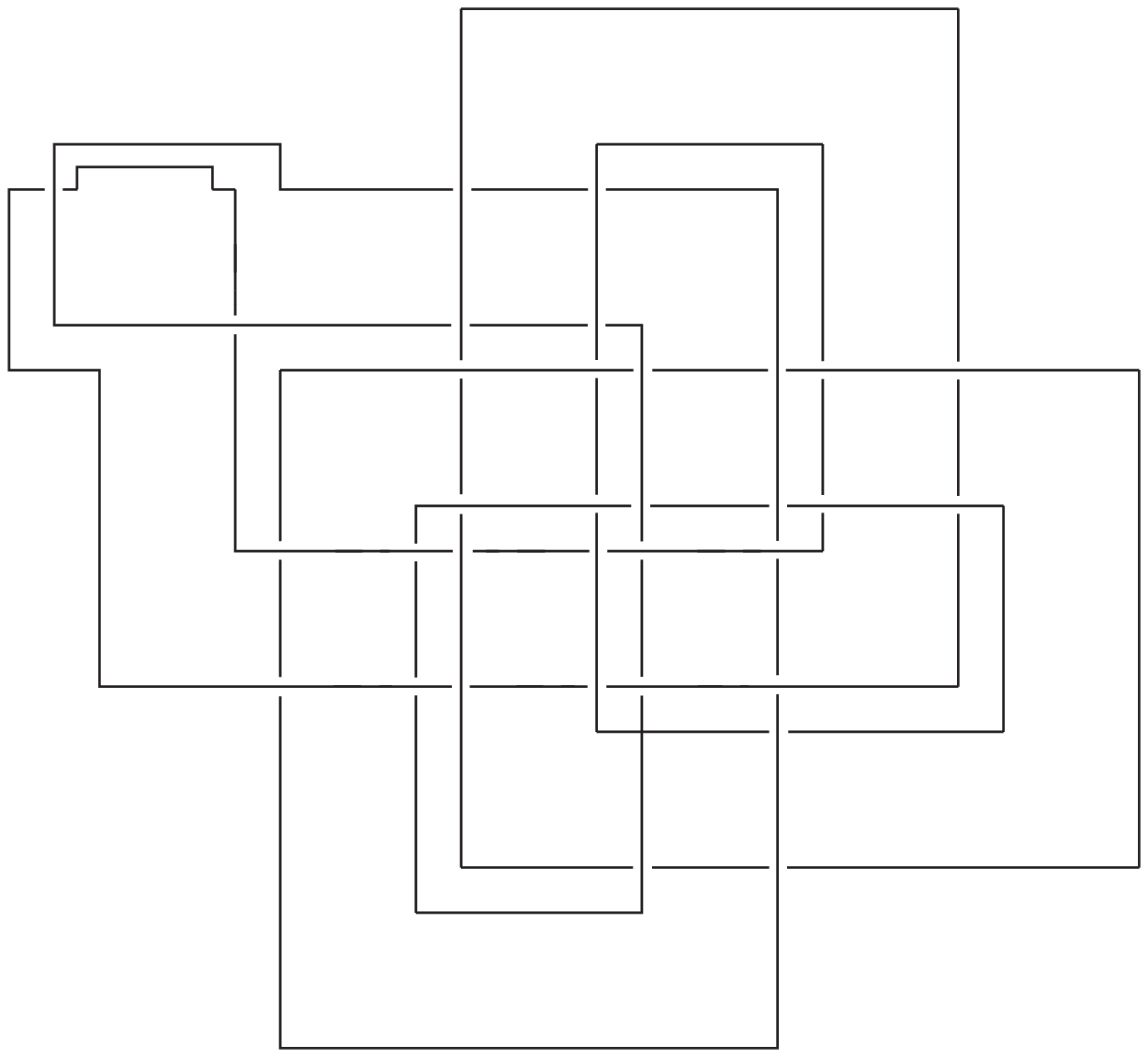}
    \end{center}
\mycap{Diagrams produced by the move $C$ of Fig.~\ref{138C:fig}.\label{138C01:fig}}
\end{figure}
are then both reduced via moves $Z_{1,2,3}$ to the trivial diagram in a handful of steps.
Another diagram for which $C$ is useful, even if not essential, is that of Fig.~\ref{trefeight:fig},
as shown in Fig.~\ref{trefeightCdisc:fig}.
\begin{figure}
    \begin{center}
    \includegraphics[scale=.6]{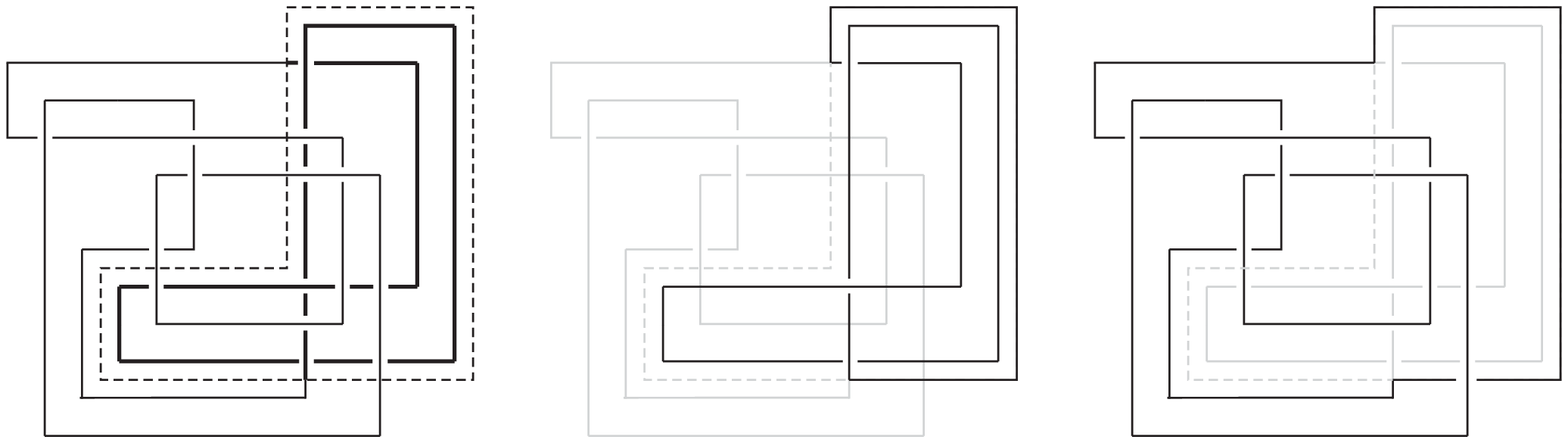}
    \end{center}
\mycap{Left: a move $C$ applies to the diagram of Fig.~\ref{trefeight:fig}.
The disc $\Omega$ is the bounded one with the dashed curve as boundary, $p(\alpha)$ is
the slightly thicker arc inside $\Omega$, and all the other arcs in $\Omega$ have label $O$.
Center and right: the diagrams produced by $C$ are the standard one of the trefoil and one
that becomes the standard one for the figure-eight after a more $R_2$.\label{trefeightCdisc:fig}}
\end{figure}

We now describe an example provided to us by the referee showing
that the move $C$ is actually essential. Figure~\ref{120:fig} displays a diagram with $120$ crossings,
\begin{figure}
    \begin{center}
    \includegraphics[scale=.6]{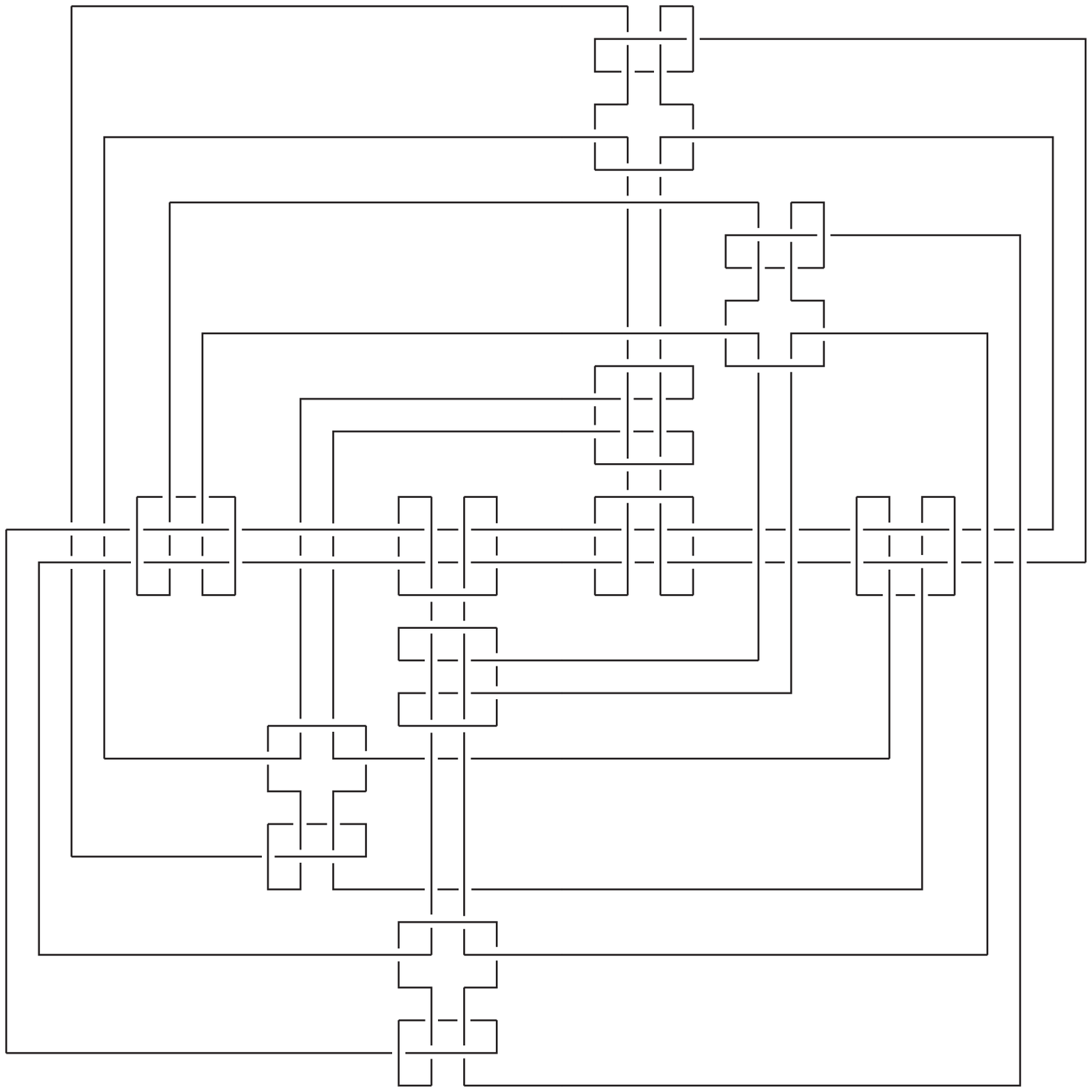}
    \end{center}
\mycap{A diagram of the unknot with 120 crossings.\label{120:fig}}
\end{figure}
which in fact represents the unknot, but our reduced algorithm
(using monotonic $Z_{1,2,3}$ moves only) was unable to untangle it.
On the other hand, our software found a $C$-move that applies to
the diagram, as shown in Fig.~\ref{120C:fig},
\begin{figure}
    \begin{center}
    \includegraphics[scale=.6]{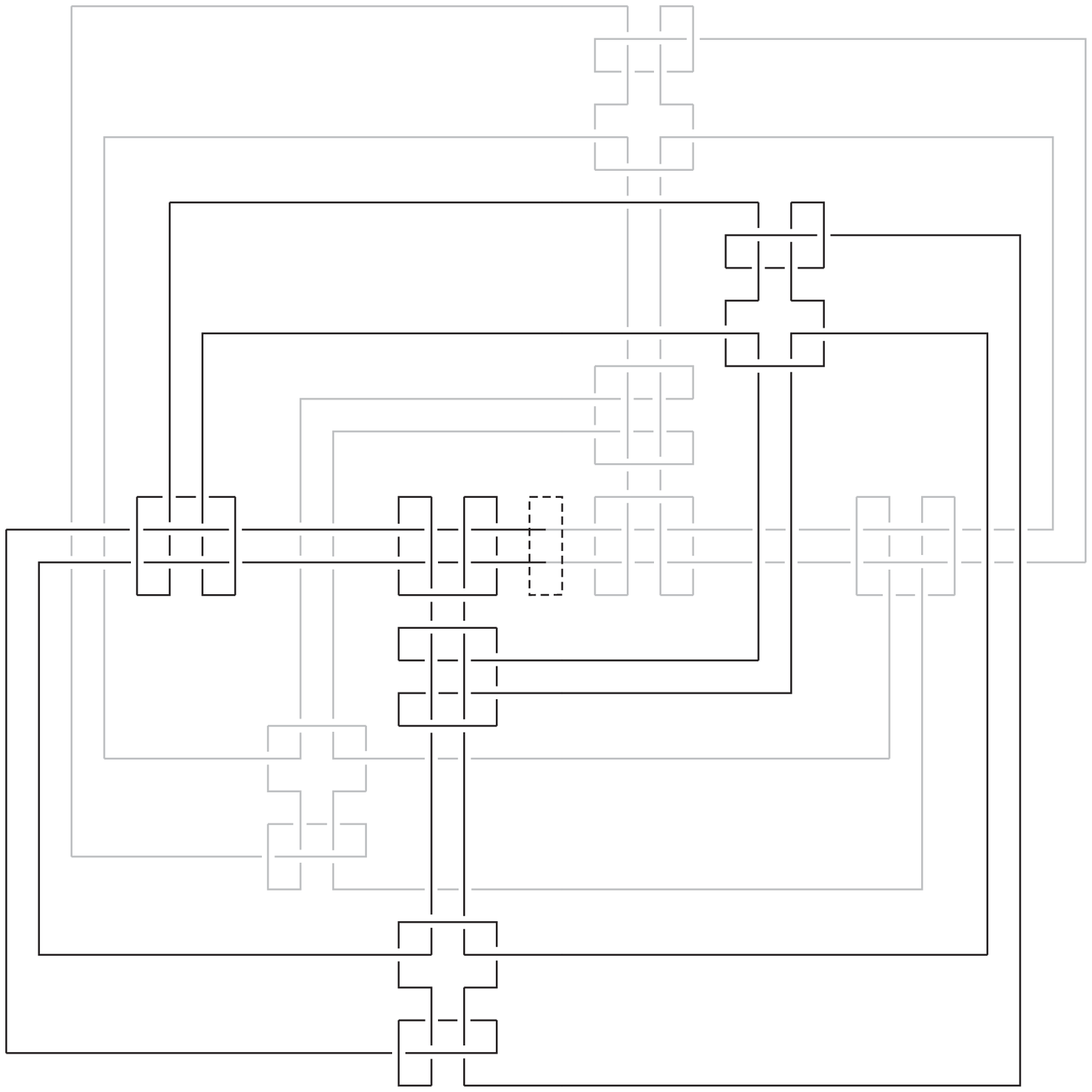}
    \end{center}
\mycap{A move $C$ applies to the diagram of Fig.~\ref{120:fig}.
The disc $\Omega$ is the unbounded one with the dashed curve as boundary,
and actually two interpretations of the move are possible: either one
declares the black arc in $\Omega$ to be $p(\alpha)$ and the gray one to have
label $U$, or one declares the gray arc to be $p(\alpha)$ and the
black one to have label $O$. The result is the same in either case.\label{120C:fig}}
\end{figure}
after which we get the diagrams of Fig.~\ref{120C01:fig},
\begin{figure}
    \begin{center}
    \includegraphics[scale=.6]{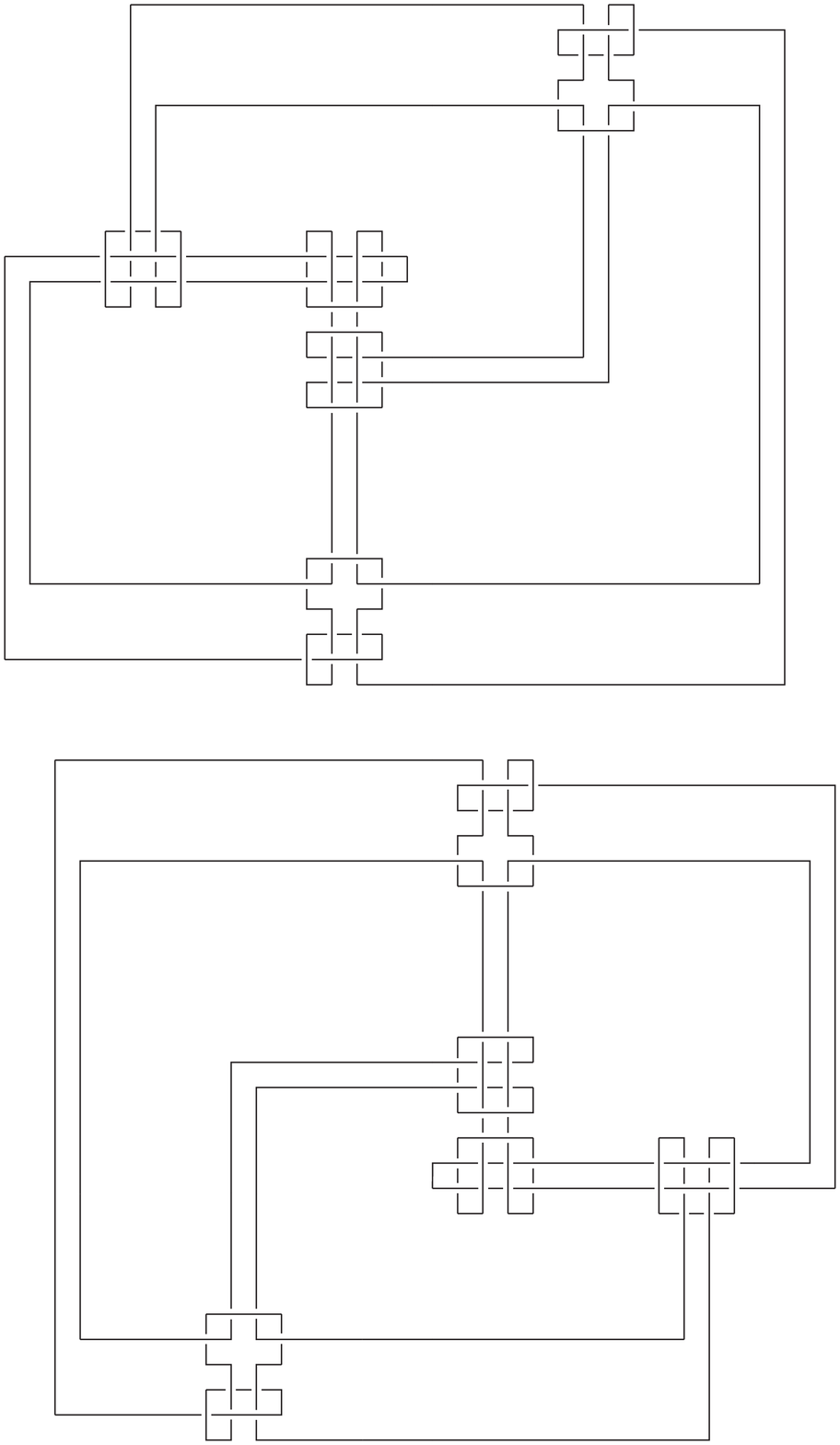}
    \end{center}
\mycap{The move $C$ of Fig.~\ref{120C:fig} gives these diagrams that untangle very easily.\label{120C01:fig}}
\end{figure}
which are both very quickly untangled via $Z_{1,2,3}$ moves.

\paragraph{Straightening of a curve}
Before we can describe the announced variant $\Ctil$ of the move $C$,
we need to give an easy definition:

\begin{defn}
\emph{If $c$ is any tame curve in the plane with normal crossings only and distinct ends, we denote by
$\sigma(c)$ and call \emph{straightening} of $c$ any simple curve contained in $c$ and having
the same ends as $c$.}
\end{defn}

\begin{rem}\label{unique:straight:rem}
\emph{$\sigma(c)$ is not unique but finitely many possibilities exist, and
one can be constructed algorithmically by following $c$ from one of its ends and erasing
any curl as soon as it is born, see Fig.~\ref{straightnew:fig}.}
\begin{figure}
    \begin{center}
    \includegraphics[scale=.6]{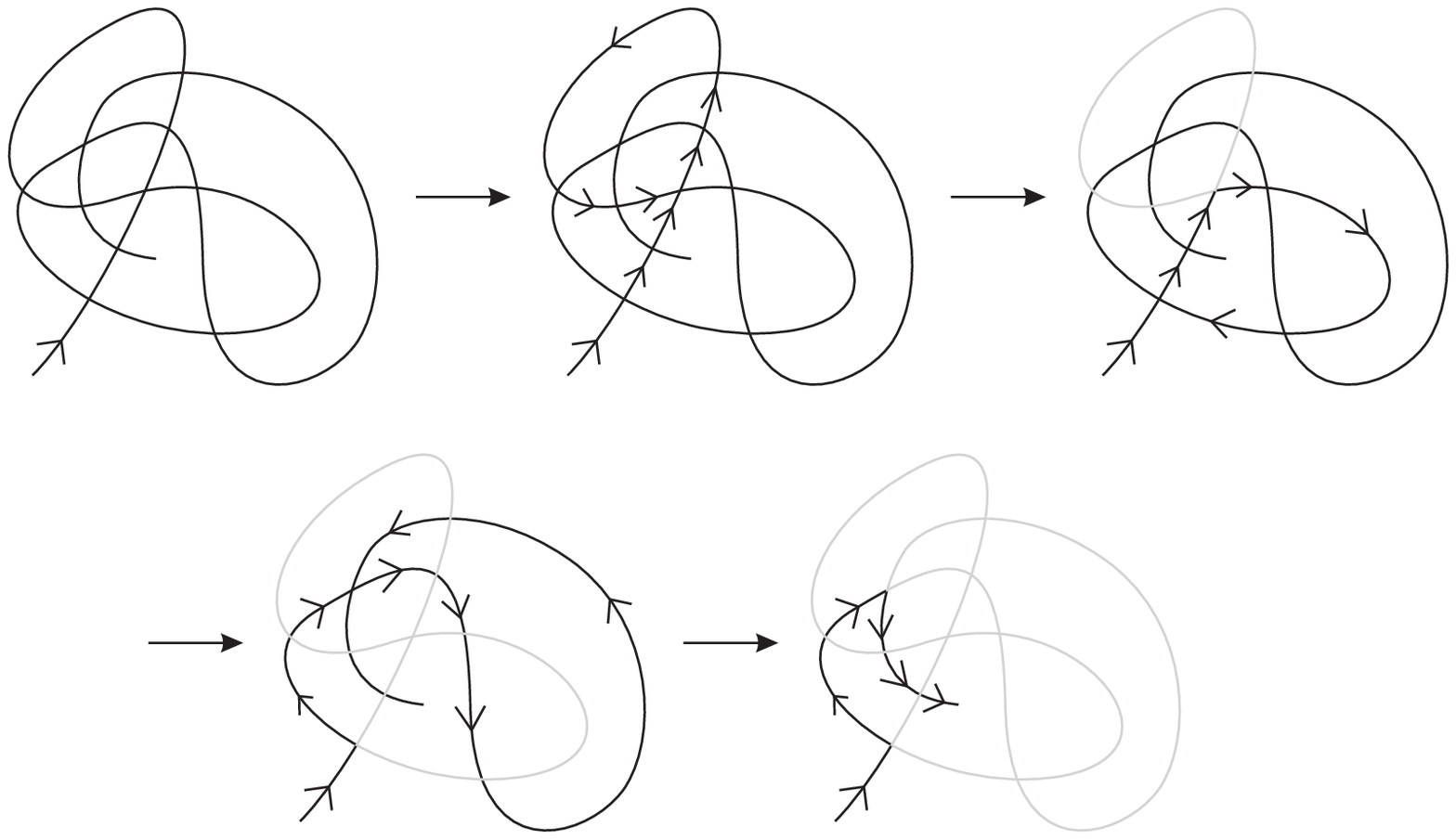}
    \end{center}
\mycap{Straightening a curve by removing curls as soon as they are born.\label{straightnew:fig}}
\end{figure}
\end{rem}

\paragraph{The move $\Ctil$}
Let $p:\mathop{\sqcup}\limits_{i=1}^p S_i^1\to S^2$ be the immersion associated to a link diagram $D$.
Let $\Omega$ and $\Omegatil$ be tame topological discs in $S^2$, with $\Omega$ contained in the interior
of $\Omegatil$, so the closure of $\Omegatil\setminus\Omega$ is an annulus $\Theta$.
Suppose $\partial\Omega$ and $\partial\Omegatil$ are transverse to $D$ and the following happens
(see Fig.~\ref{preCminus:fig}):
\begin{figure}
    \begin{center}
    \includegraphics[scale=.5]{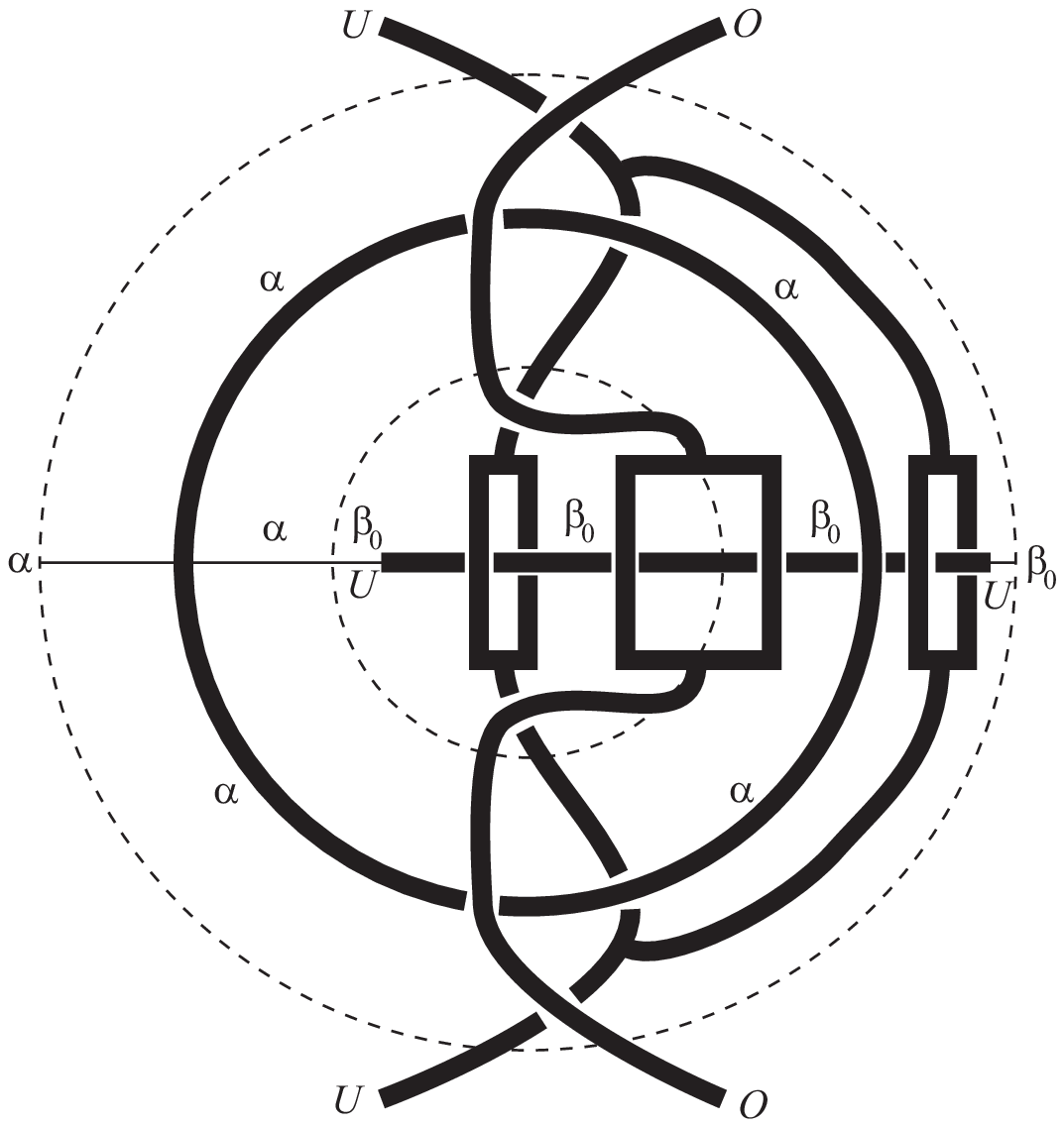}
    \end{center}
\mycap{A diagram $D$ to which the move $\Ctil$ applies.\label{preCminus:fig}}
\end{figure}\begin{itemize}
\item There exists an arc $\alpha\subset S_i^1$
such that $p(\alpha)\subset\Theta$ has one end on each component of $\partial\Theta$;
\item The preimage $p^{-1}\left(\Omegatil\right)$ has components
$\betaztil,\beta_1,\ldots,\beta_N$ with $\betaztil\supset\alpha$;
\item If $\beta_0=\betaztil\setminus\alpha$ one can choose labels $\lambda_0,\lambda_1,\ldots,\lambda_N$ in $\{O,U,\}$ so that:
\begin{itemize}
\item If $\lambda_j=O$ then $p(\beta_j)$ is over $p(\alpha)$ where they cross;
\item If $\lambda_j=U$ then $p(\beta_j)$ is under $p(\alpha)$ where they cross;
\item If $\lambda_j=O$ and $\lambda_k=U$ then $p(\beta_j)$ is over $p(\beta_k)$ where they cross.
\end{itemize}
\end{itemize}
Given these data, we call $C$ the move that replaces $D$ by the pair of diagrams $(D_0,D_1)$, where (see Fig.~\ref{postCminus01:fig}):
\begin{itemize}
\item $D_0$ is $p(\alpha)\cup\sigma(\beta_0)$ union one of the two halves into which the ends of $\beta_0$
split $\partial\Omegatil$, with all the crossings as in $D$;
\item $D_1$ is $\left(D\setminus p(\alpha)\right)\cup\sigma(p(\alpha))$, with all the crossings
as in $D$.
\end{itemize}
\begin{figure}
    \begin{center}
    \includegraphics[scale=.5]{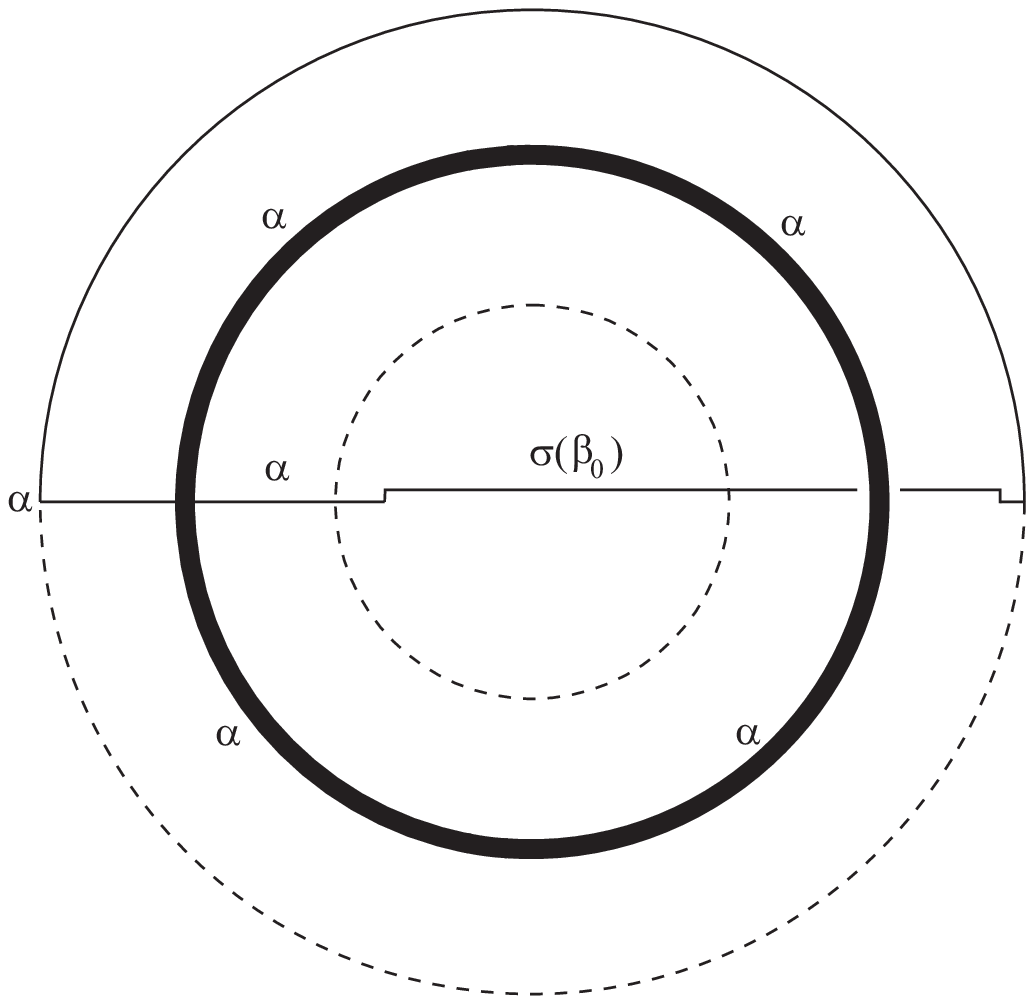}\qquad
    \includegraphics[scale=.5]{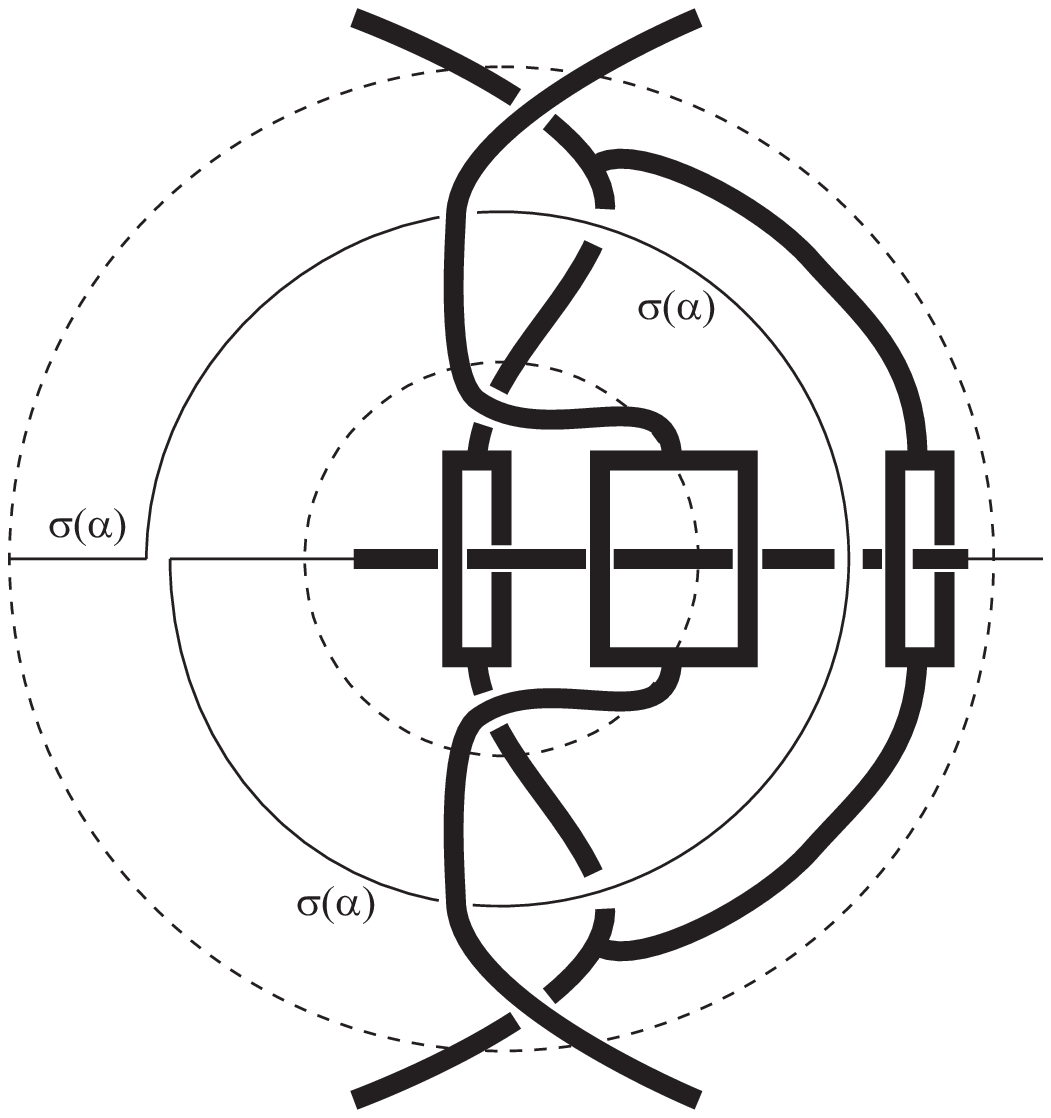}
    \end{center}
\mycap{The diagrams $D_0$ and $D_1$.\label{postCminus01:fig}}
\end{figure}

\begin{prop}
$[D]=[D_0]\#[D_1]$.
\end{prop}

\begin{proof}
Suppose the label $\lambda_0$ of $\beta_0$ is $U$.
The assumptions then imply that $[D]$ can be realized as a link $L$ in $\matR^3$ such that
$L\cap \left(\Omegatil\times\matR\right)$ is the disjoint union of 4 portions:
\begin{itemize}
\item One in $\Theta\times(-\varepsilon,\varepsilon)$, with projection $p(\alpha)$;
\item One in $\Omegatil\times(1-\varepsilon,1+\varepsilon)$, with projection $\bigcup\{p(\beta_i):\ \lambda_i=O\}$;
\item One in $\Omegatil\times(-1-\varepsilon,-1+\varepsilon)$, with projection $\bigcup\{p(\beta_i):\ \lambda_i=U\}$;
\item A vertical arc $P\times[-1,0]$ with $P\in\partial\Omega$.
\end{itemize}
So the sphere $\partial(\Omega\times[-\varepsilon,\varepsilon])$ meets $L$
transversely at two points, whence it allows to express $L$ as $L_0\# L_1$, and of course $L_j=[D_j]$.
\end{proof}

\begin{rem}
\begin{itemize}
\item \emph{In $D_0$ we used $\sigma(\beta_0)$ to join the ends of $p(\alpha)$ to have
a definite procedure, but any simple arc joining them in $\Theta$ would work, provided its crossings with $p(\alpha)$
have the same type as those of $\beta_0$;}
\item \emph{Similarly, in $D_1$ we could replace $\sigma(p(\alpha))$ with any simple arc in $\Theta$ having the same ends,
provided its crossings with the $\beta_i$'s have the same type as those of $p(\alpha)$.}
\end{itemize}
\end{rem}

An example of move $\Ctil$ is shown in Figg.~\ref{CminusExample:fig}
\begin{figure}
    \begin{center}
    \includegraphics[scale=.45]{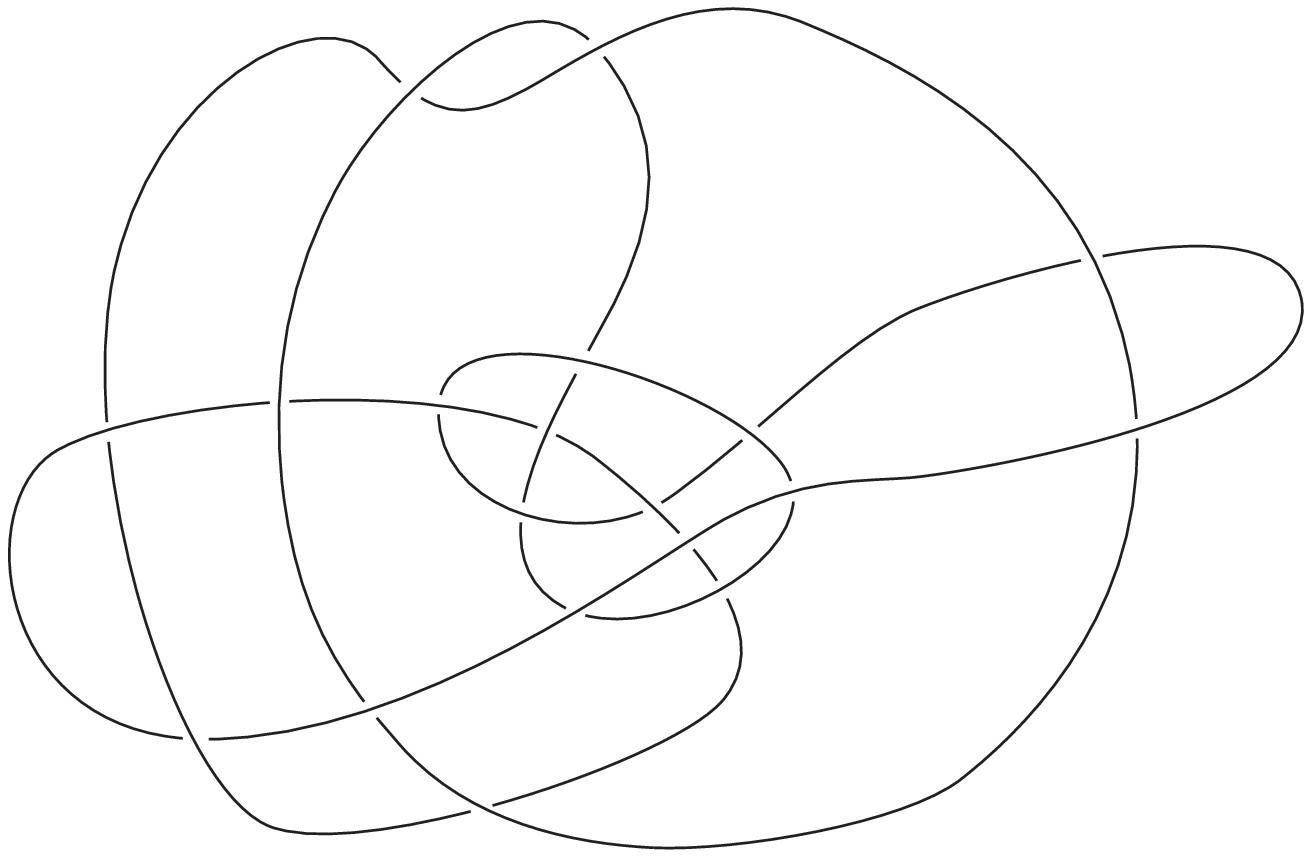}\qquad
    \includegraphics[scale=.45]{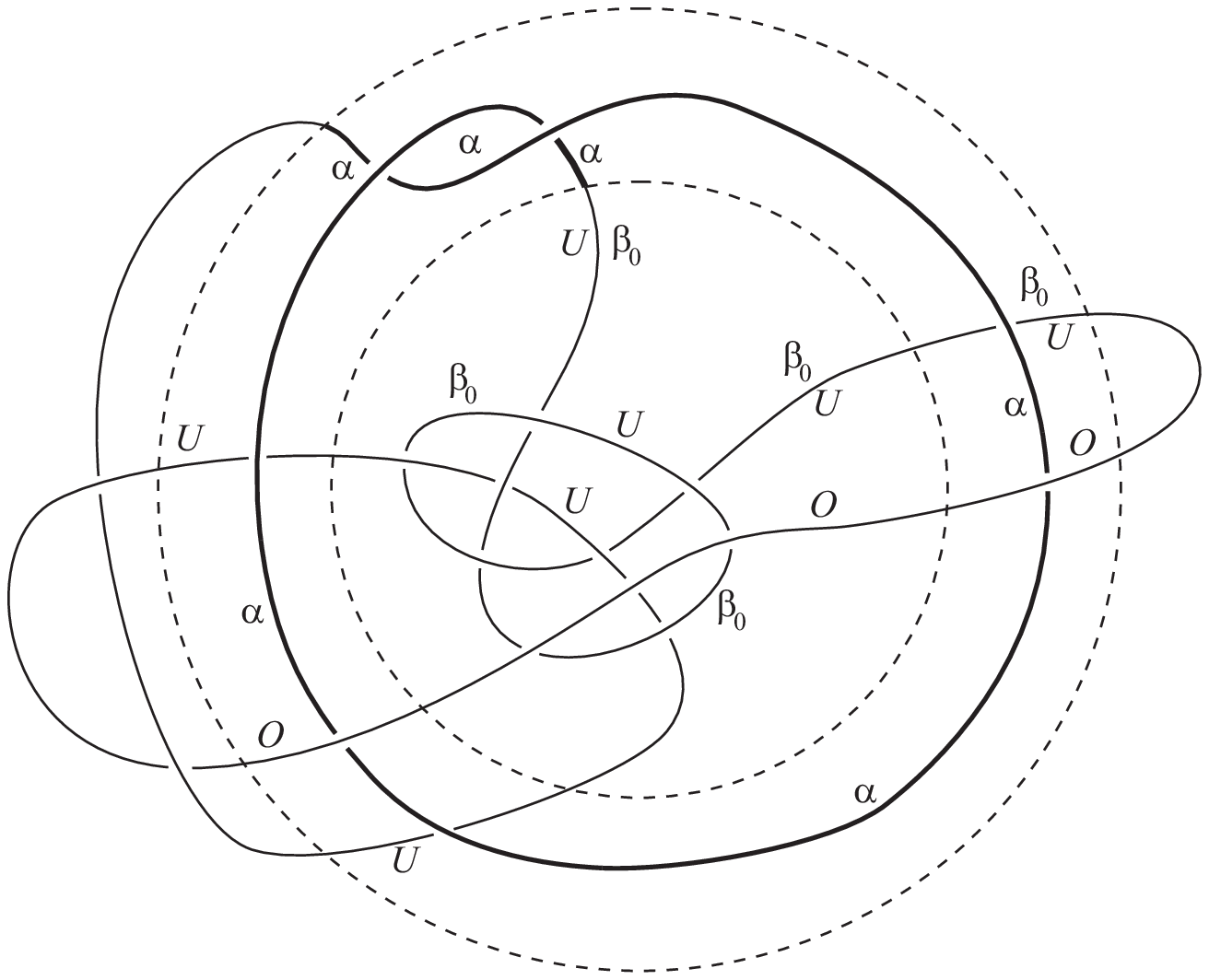}
    \end{center}
\mycap{A diagram and the identification of a move $\Ctil$ that applies to it.\label{CminusExample:fig}}
\end{figure}
and~\ref{CminusExample01:fig}; note that no move $C$ exists in this case.
\begin{figure}
    \begin{center}
    \includegraphics[scale=.45]{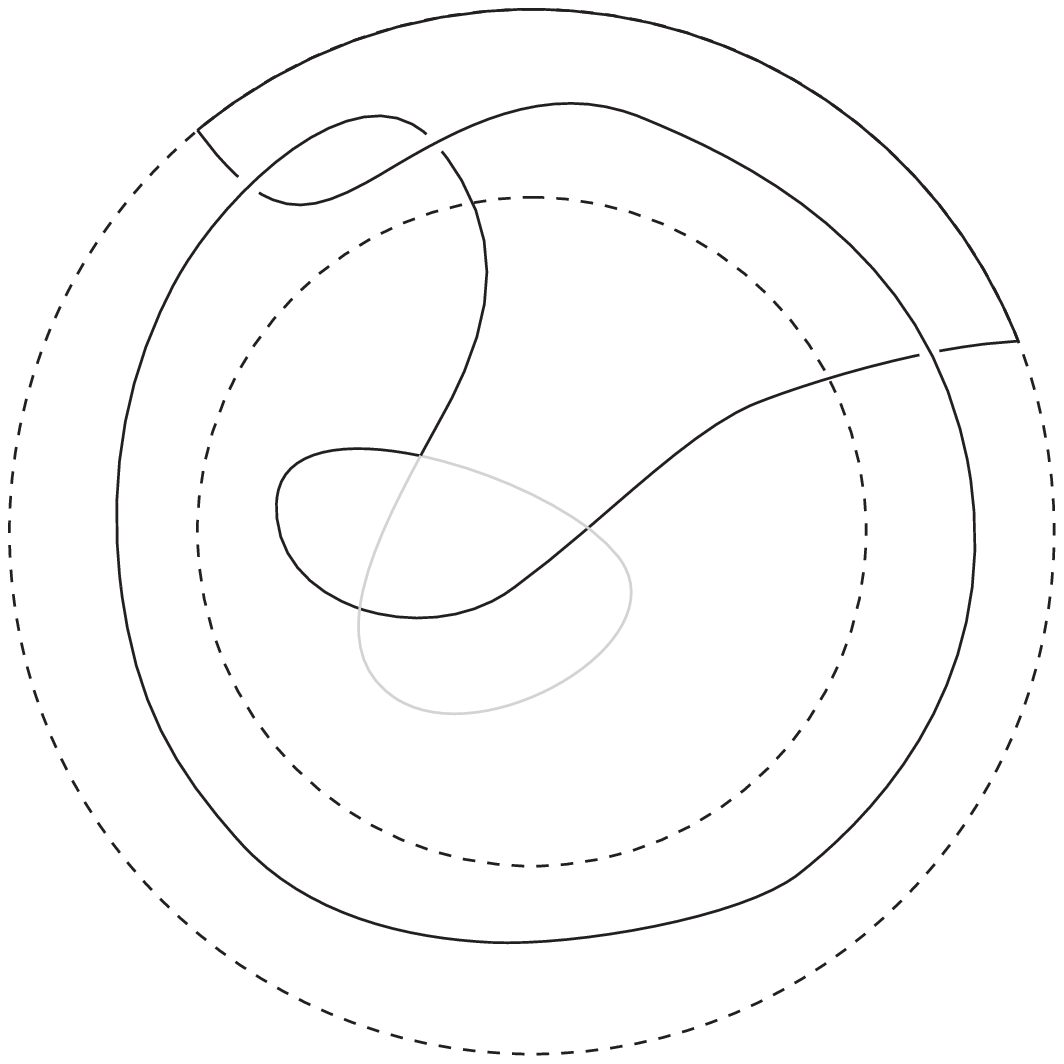}\qquad
    \includegraphics[scale=.45]{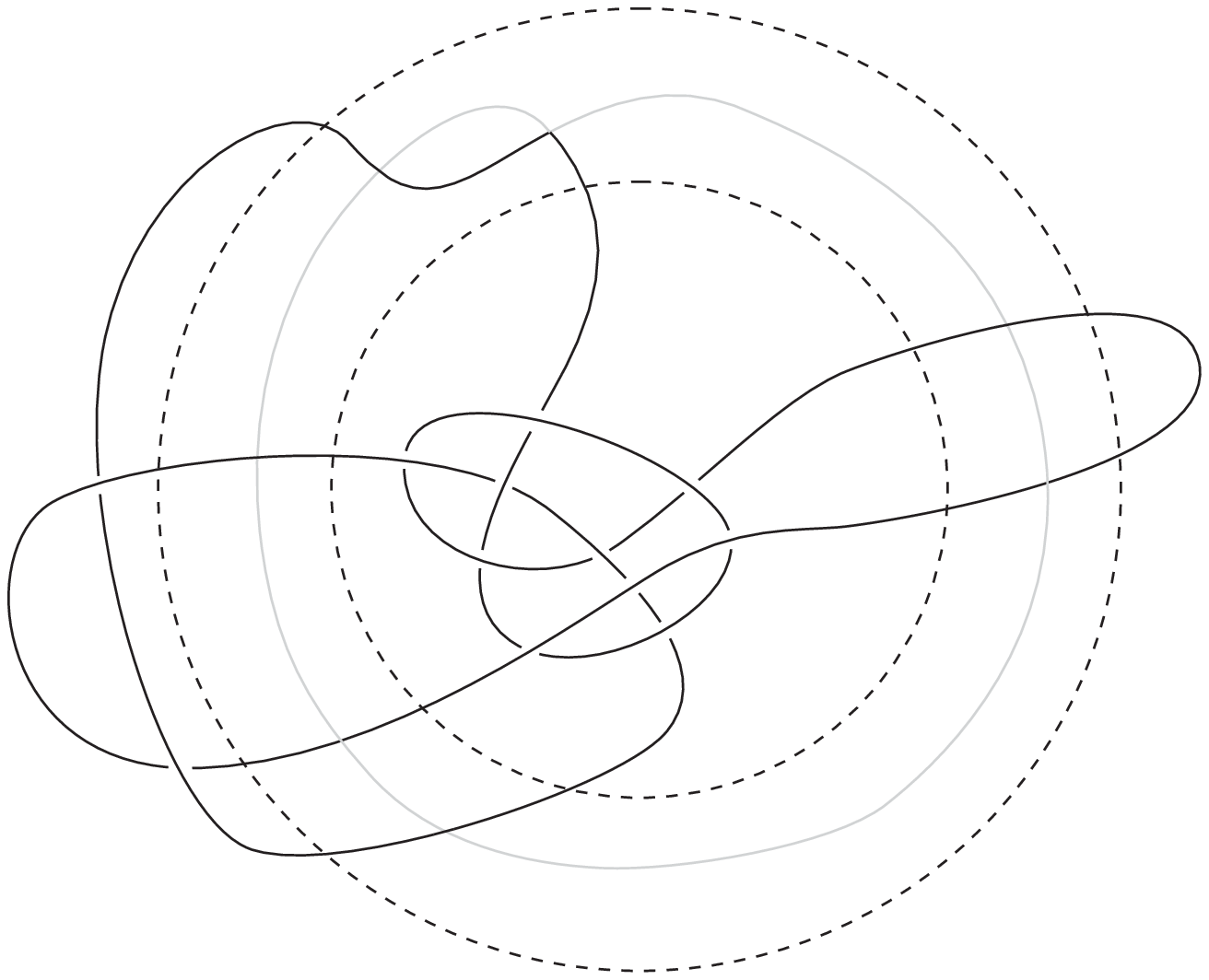}
    \end{center}
\mycap{Diagrams resulting from the move $\Ctil$.\label{CminusExample01:fig}}
\end{figure}
As for the move $C$, the construction implies that $c(D_0)+c(D_1)\leqslant c(D)$. Moreover,
$c(D_0)<c(D)$ and $c(D_1)<c(D)$ except in situations,
that one can easily avoid, where one of the $D_i$'s is actually $D$.

We conclude by remarking that we do not know whether $\Ctil$ is actually essential:
all the the knot diagrams we have used so far as tests for our algorithm were completely simplified
in a monotonic fashion using the moves $Z_{1,2,3}$ and $C$ only.
(For the diagram of Fig.~\ref{CminusExample:fig} a $C$ appears after some $Z_{1,2,3}$.)

\vspace{.5cm}

\noindent
adolfo@zanellati.it

\vspace{.5cm}

\noindent
Dipartimento di Matematica\\
Universit\`a di Pisa\\
Largo Bruno Pontecorvo, 5\\
56127 PISA -- Italy\\
petronio@dm.unipi.it

\end{document}